\documentclass[11pt, a4paper]{amsart}
\usepackage[left=1in,right=1in,top=1in,bottom=1in]{geometry}

\usepackage[all,cmtip]{xy}
\usepackage{amssymb}
\usepackage[english]{babel}
\usepackage{mathrsfs}
\usepackage{enumerate}
\RequirePackage{hyperref}
\usepackage{rotating}
\usepackage{amsmath}
\usepackage{comment}

\usepackage[dvipsnames]{xcolor}

\input xy
\xyoption{all}
\newcommand{\leqdr}{\mathbin{\rotatebox[origin=c]{-45}{$\supseteq$}}}
\newcommand{\lequr}{\mathbin{\rotatebox[origin=c]{45}{$\supseteq$}}}
\newcommand{\geqdr}{\mathbin{\rotatebox[origin=c]{-45}{$\subseteq$}}}
\newcommand{\gequr}{\mathbin{\rotatebox[origin=c]{45}{$\subseteq$}}}
\newcommand{\leqd}{\mathbin{\rotatebox[origin=c]{-90}{$\supseteq$}}}

\newdir{d}{\leq}
\newdir{d}{\supseteq}
\newdir{dr}{\leqdr}
\newdir{ur}{\lequr}

\newtheorem{theorem}{Theorem}[section]
\newtheorem{lemma}[theorem]{Lemma}

\newtheorem{proposition}[theorem]{Proposition}
\newtheorem{corollary}[theorem]{Corollary}

\theoremstyle{definition}
\newtheorem{definition}[theorem]{Definition}

\newtheorem{example}[theorem]{Example}

\newtheorem{remark}[theorem]{Remark}

\newtheorem{question}[theorem]{Question}

\newtheorem{fact}[theorem]{Fact}

\def\diam{\mathrm{diam}}
\def\R{\mathbb{R}}

\def\MM{\mathcal M}
\def\QH{\mathcal{QH}}
\def\HH{\mathcal H}
\def\BB{\mathcal B}
\def\PN{\mathcal{PN}}
\def\OB{\mathcal{CB}}
\def\OrB{\mathcal{OB}}
\def\KK{\mathcal K}
\def\XX{\mathcal X}
\def\YY{\mathcal Y}
\def\AA{\mathcal A}
\def\BB{\mathcal B}
\def\II{\mathcal I}
\def\VV{\mathcal V}
\def\FF{\mathcal F}
\def\E{\mathcal E}
\def\WPN{\mathcal{WPN}}
\def\TopGrp{{\bf TopGrp}}
\def\Set{{\bf Set}}
\def\cl{\mathfrak{cl}}
\def\UC{\mathcal{UC}}
\def\FH{\mathcal{FH}}
\def\UB{\mathcal{UB}}
\def\DD{\mathcal D}

\def\FGAb{{\bf FGA}}
\def\FG{{\bf FG}}
\def\countAb{{\bf \omega\mbox{-}Ab}}
\def\countG{{\bf \omega\mbox{-}Grp}}
\def\Ab{{\bf Ab}}
\def\Grp{{\bf Grp}}
\def\sigmacAb{{\bf \sigma c\mbox{-}LCA}}
\def\sigmac{{\bf \sigma c\mbox{-}LC}}
\def\LCA{{\bf LCA}}
\def\LC{{\bf LC}}

\DeclareMathOperator{\Act}{Act}
\DeclareMathOperator{\Cl}{Cl}

\usepackage{color}
\newenvironment{revty}{\color{red}}{}
\DeclareMathOperator{\Symd}{Sym(\mathbb{N})}

\author[D. Shakhmatov]{Dmitri Shakhmatov}
\address{Division of Mathematics, Physics and Earth Sciences, Graduate School of Science and Engineering, Ehime University, Matsuyama, 790-8577, Japan}
\email{shakhmatov.dmitri.mx@ehime-u.ac.jp}

\author[T. Yamauchi]{Takamitsu Yamauchi}
\address{Division of Mathematics, Physics and Earth Sciences, Graduate School of Science and Engineering, Ehime University, Matsuyama, 790-8577, Japan}
\email{yamauchi.takamitsu.ts@ehime-u.ac.jp}

\author[N. Zava]{Nicol\`o Zava$^\ast$}
\address{Institute of Science and Technology Austria (ISTA), 3400 Klosterneuburg, Austria}
\email{nicolo.zava@gmail.com}
\thanks{The third named author was supported by the FWF Grant, Project number I4245-N35}

\title{Coarse structures on locally compact groups}

\begin{document}

\begin{abstract}
Motivated by the study of the large-scale geometry of topological groups, we investigate particular families of subsets of topological groups named group ideals. We compare different group ideals in the realm of locally compact groups. In particular, we show that a subset of a locally compact abelian group is relatively compact if and only if it is coarsely bounded. Using this result, we prove that an infinite-dimensional Banach space cannot be embedded into any product of locally compact groups. 
\end{abstract}

	
\maketitle	


\section{Introduction}

Coarse geometry, also known as large-scale geometry, is the study of large-scale properties of spaces, ignoring their local, small-scale ones. More precisely, two metric spaces 
$(X_1,d_1)$ and $(X_2,d_2)$
 are considered equivalent in this theory if they are coarsely equivalent, i.e., there are two maps $f_1\colon X_1\to X_2$ and $f_2\colon X_2\to X_1$ such that:
\begin{enumerate}[$\bullet$]
\item for every $i\in\{1,2\}$ and every $R_i\in\R_{\geq 0}$, there exists $S_i\in\R_{\geq 0}$ such that, for every $x,y\in X_i$, 
$d_j(f_i(x),f_i(y))\leq S_i$ provided that $d_i(x,y)\leq R_i$ (where $\{i,j\}=\{1,2\}$);
\item 
$\max\{\sup\{d_2(f_1\circ f_2(x),x)\mid x\in X_2\},\sup\{d_1(f_2\circ f_1(x),x)\mid x\in X_1\}\}<\infty$.
\end{enumerate}	
This approach has been particularly fruitful in the branch of algebra known as geometric group theory since the breakthrough work of Gromov. The key observation is the following. If $\Sigma$ is a finite generating set of a finitely generated group $G$, then $G$ can be equipped with the word metric $d_\Sigma$ associated to $\Sigma$. While these metrics on $G$ can be very different, for every pair of finite generating sets $\Sigma$ and $\Delta$, the identity map $id\colon(G,d_\Sigma)\to(G,d_\Delta)$ is a coarse equivalence. Hence every finitely generated group has essentially just one inner metric from the large-scale point of view. We address \cite{harpe} for a wide discussion of this topic. Let us also mention that Yu proved in \cite{yu} that the Novikov conjecture holds for finitely generated groups of finite homotopy type with finite asymptotic dimension.

In order to generalise this fruitful metric approach to more general classes of (topological) groups, some steps have been made. In particular, every countable group can be endowed with a left-invariant proper (i.e., whose closed balls are compact) metric. Moreover, each pair of such metrics are coarsely equivalent and the identity map shows that (\cite[Proposition 1]{smith}, see also \cite{struble} where the author proved that every second countable locally compact group can be endowed with a left-invariant proper metric
). In the particular case of finitely generated groups, the word metrics are the desired left-invariant proper metrics. Among the studies of these groups, let us just cite the papers \cite{DraSmi}, where some important theorems concerning the asymptotic dimension are proved, and \cite{banakh_highes_zarichnyi}, which provides a complete classification of countable abelian groups up to coarse equivalence.

A further step into generalisation was provided by Cornulier and de la Harpe (see their monograph \cite{cornulier_harpe} for a comprehensive discussion). They noticed that in the realm of $\sigma$-compact locally compact groups the metric approach can be generalised. In fact, every $\sigma$-compact locally compact group has a left-invariant proper pseudometric that is locally bounded (i.e., every point has a neighbourhood of finite diameter), and every pair of such pseudometrics are coarsely equivalent.

In order to go beyond 
that case,
the metric approach is not enough and it is necessary to introduce some special coarse structures on groups and topological groups. Coarse structures were introduced by Roe (see, for example, his monograph \cite{roe}) to encode and then generalise the large-scale properties of metric spaces. Coarse structures are defined as the large-scale counterpart of the classical notion of uniformity (
\cite{weil}, \cite{isbell}) that capture and extend the small-scale properties of metric spaces. As uniform structures are a useful tool to parametrise topological groups, one can use coarse structures to study the so-called coarse groups. The first formalisation of these objects using coarse structures can be found in \cite{nicas_rosenthal}, however, previously Protasov and Protasova (\cite{protasov_protasova}) developed the same theory for balleans, which are a different, although equivalent, formalisation of coarse spaces (see \cite{protasov_banakh} for the definition and \cite{protasov_zarichnyi,dikranjan_zava_cat} for the equivalence between balleans and coarse spaces). Following those references, a coarse group is a group endowed with a coarse structure that, since it agrees with the algebraic structure of the group, can be characterised by the use of a particular ideal of subsets of the group, called group ideal. This idea is the counterpart of how group topologies are characterised by the filter of neighbourhood of the identity. Let us also mention that the term coarse group appears in the literature also for group objects in the category of coarse spaces (see \cite{LeiVig}, where the authors introduce and study this notion).

While on abstract groups the group ideals induced by the cardinality (in particular the group ideal consisting of all finite subsets) seem to be the right choice (see, for example, \cite{protasov,dikranjan_zava_qhomo}, although in the last paper also some other choices were also proposed), as for topological groups the situation is much more unclear (see \cite{nicas_rosenthal,rosendal,dikranjan_zava_Pon,dikranjan_protasov}). Among all the possible choices, two seem to be particularly fruitful, namely the group ideal $\KK$ of relatively compact subsets (studied for example in \cite{nicas_rosenthal,cornulier_harpe,dikranjan_zava_Pon}) and the one, $\OB$, of 
coarsely bounded subsets
introduced by Rosendal (\cite{rosendal}). The latter has the advantage, among the others, that agree with the norm of Banach spaces, i.e., the 
coarsely bounded subsets
of a Banach space are precisely those that are norm-bounded
(\cite[Corollary 2.20]{rosendal}).
However, with the group ideal of relatively compact subset in the realm of locally compact abelian groups nice results connecting topological notions with their large-scale counterpart through the Pontryagin duality functor (see, for example, \cite[Chapter 6]{pontryagin} 
or \cite[Chapter 13]{AusDikGio}
for details about this duality) have been found. In this direction, let us cite that the Pontryagin functor transforms the asymptotic dimension into the covering dimension (\cite{nicas_rosenthal_asdim_dim}), metrisability of the coarse structure into metrisability of the topology (\cite{dikranjan_zava_Pon}), and, in the discrete case, the coarse entropy of a surjective morphism into the topological entropy of its dual (\cite{zava_entropy}).

In general, $\KK$ is a subfamily of $\OB$ (see Definition \ref{def:K_and_OB}).
It has been proved by Rosendal 
\cite[Corollary 2.19]{rosendal}
that $\KK$ and $\OB$ agree in the realm of $\sigma$-compact locally compact groups, and thus considering the left-coarse structure (i.e., the one induced by the group ideal $\OB$) extends the results of Cornulier and de la Harpe (\cite{cornulier_harpe}). Moreover
in \cite{rosendal}, Rosendal proved that they dramatically differ both for arbitrary locally compact groups (actually discrete groups%
, see \cite[Example 2.26]{rosendal}%
) and for (even abelian) Polish groups
(such as infinite dimensional separable Banach spaces, see \S \ref{sec:embBanach}).
However, not much was known except for these results. In particular, it is natural to ask whether they coincide in the realm of locally compact abelian groups, which are not in general $\sigma$-compact. 
In this paper, we prove the following (see Corollary \ref{coro:LCA}).
\begin{theorem}
\label{thm:K_LCA_OB}
For any locally compact abelian group, the family $\OB$ of coarsely bounded subsets coincides with the family $\KK$ of relatively compact subsets.
\end{theorem}
Applying Theorem \ref{thm:K_LCA_OB}, we show that no infinite-dimensional Banach space can be embedded in any product of locally compact groups 
(Theorem \ref{thm:Banach_space_not_in_prod_lc}), 
and we can answer \cite[Question 5.6]{dikranjan_zava_Pon} on metrisability of the left-coarse struture of the dual group of a locally compact abelian group (\S \ref{sub:question_pon}).

Inspired by the characterisations of the group ideal $\OB$, we also
introduce several group ideals on topological groups (Definitions \ref{def:group_ideals_from_maps} and \ref{def:group_ideals_from_actions}). We then study the relationships between these new group ideals and others that were already considered in the literature (Definitions \ref{ex:group_c_s}, \ref{def:K_and_OB} and \ref{def:group_ideals_V_and_F}). In particular, we evaluate those group ideals in subclasses of locally compact groups. 
For the classes of topological abelian groups that we consider, the relationships between the group ideals is completely understood.

The paper is organised as follows. The needed preliminary notions of coarse geometry are collected in Section \ref{sec:preliminaries}. In particular, \S\ref{sub:cs} contains the basic definitions of coarse spaces and morphisms between them, while \S\S\ref{sub:cg} and \ref{sec:group_ideals} concern coarse groups and the characterisation of coarse properties using group ideals. Section \ref{sec:group_ideals_topgrp} is devoted to the definition of the group ideals on topological groups we consider in this paper. In \S\ref{sub:functorial_gi_from_maps} and in \S\ref{sub:functioral_gi_by_actions} we introduce group ideals induced by families of maps and actions, respectively. We study their relationships in Section \ref{sec:relationships_with_no_further_conditions}. The next part of the paper is devoted to investigate the group ideals in subclasses of locally compact groups. In order to do that, in Section \ref{sec:stability} we develop tools concerning in particular products and quotients. In Section \ref{sec:group_ideals_TopGrp} we study how the topology of subclasses of locally compact groups (of locally compact abelian groups in \S\ref{sub:group_ideals_TopAbGrp}) impacts on the group ideals%
, and prove Theorem \ref{thm:K_LCA_OB}.
In Section \ref{sec:Appl}
we present the applications of Theorem \ref{thm:K_LCA_OB} stated above.
We summarise the obtained results 
on group ideals
in Section \ref{sec:Hasse_group_ideals}, describing the situation in each subclass separately and providing, in \S\ref{sub:counterexamples_questions}, the needed examples and open questions. 

\section{Preliminaries}\label{sec:preliminaries}

\subsection{Coarse spaces}\label{sub:cs}

For a set $X$ let $\mathcal{P}(X)$ denote the power set of $X$.

\begin{definition}\label{Def:Roe} According to Roe (\cite[%
Definiiton 2.3%
]{roe}), a \emph{coarse space} is a pair $(X,\E)$, where $X$ is a set and $\E\subseteq\mathcal P(X\times X)$ a \emph{coarse structure} on it, which means that
	\begin{enumerate}[(i)]
		\item $\Delta_X=\{(x,x)\mid x\in X\}\in\E$;
		\item $\E$ is an {\em ideal}, i.e., it is closed under taking finite unions and subsets;
		\item if $E\in\E$, then $E^{-1}=\{(y,x)\in X\times X\mid(x,y)\in E\}\in\E$;
		\item if $E,F\in\E$, then $E\circ F=\{(x,y)\in X\times X\mid\exists z\in X:\,(x,z)\in E,(z,y)\in F\}\in\E$. 
	\end{enumerate}
	An element $E$ of $\E$ is called {\em entourage}.
\end{definition}

If $X$ is a set, a {\em base of a coarse structure} is a family 
$\DD$
of entourages such that its {\em completion} 
$\cl(\DD)=\{F\subseteq X\times X \mid \exists  D \in \DD  : F \subseteq D\}$
is a coarse structure. 

Let us give one of the main examples of coarse structures.
\begin{example}
A {\em pseudometric space} is a pair $(X,d)$ of a set $X$ and a {\em pseudometric} $d\colon X\times X\to\R_{\geq 0}$ satisfying:
\begin{enumerate}[$\bullet$]
\item $d(x,x)=0$, for every $x\in X$;
\item $d(x,y)=d(y,x)$, for every $x,y\in X$;
\item $d(x,y)\leq d(x,z)+d(z,y)$, for every $x,y,z\in X$.
\end{enumerate}
Let $(X,d)$ be a pseudometric space. For every $R>0$, we define a particular subset of the square $X\times X$, called the {\em strip of width $R$}, as follows:
\begin{equation*}
S_R=\bigcup_{x\in X}\bigg(\{x\}\times B(x,R)\bigg)\subseteq X\times X,
\end{equation*}
where $B(x,R)$ denotes the closed ball centred in $x$ with radius $R$. Then the family $\{S_R\mid R>0\}$ is a base of the so-called {\em metric-coarse structure} $\mathcal E_d$.
\end{example}

If $(X,\E)$ is a coarse space, a subset $B$ of $X$ is {\em bounded} if there exists an entourage $E\in\E$ such that $B\subseteq E[x]=\{y\in X\mid (x,y)\in E\}$ for every point $x\in B$. 
%
%

If $(X,d)$ is a pseudometric space, a subset $A$ of $X$ is bounded in $(X,\E_d)$ if and only if it has finite diameter. We denote the diameter of $A$ by $\diam_d(A)$, or by $\diam(A)$ if the pseudometric is clear.

Let $(X,\E)$ be a coarse space and $x\in X$ be a point. A coarse space $(X,\E)$ is {\em connected} if $\bigcup\E=X\times X$.

We say that two maps $f,g\colon S\to(X,\E)$ from a set to a coarse space are \emph{close}, and we write $f\sim g$, if $\{(f(x),g(x))\mid x\in S\}\in\E$.

\begin{remark}\label{rem:bounded_subsets_close_maps}
Let $f,g\colon S\to(X,d)$ be two close map from a non-empty set to a metric space. Then, for every subset $A\subseteq S$, $\diam(f(A))<\infty$ if and only if $\diam(g(A))<\infty$. In fact, if $R\geq 0$ satisfies $d(f(x),g(x))\leq R$, for every $x\in S$, then $\diam(f(A))\leq \diam(g(A))+2R$ and $\diam(g(A))\leq\diam(f(A))+2R$.
\end{remark}

\begin{definition}\label{def:Mor}
	Let $(X,\E_X)$ and $(Y,\E_Y)$ be two coarse spaces. A map $f\colon X\to Y$ is
	\begin{enumerate}[(i)]
		\item \emph{bornologous} if $(f\times f)(E)\in\E_Y$ for all $E\in\E_X$;
		\item {\em large-scale injective} if $(f\times f)^{-1}(\Delta_Y)\in\E_X$;
\item {\em large-scale surjective} if $f(X)$ is large in $Y$.
		\item {\em weakly uniformly bounded copreserving} (\cite{Za}) if, for every $E\in\E_Y$, there exists $F\in\E_X$ such that $E\cap(f(X)\times f(X))\subseteq(f\times f)(F)$;
		\item {\em effectively proper} if  $(f\times f)^{-1}(E)\in\E_X$ for all $E\in\E_Y$;
		\item a \emph{coarse equivalence} if $f$ is bornologous and one of the following equivalent conditions holds:
		\begin{enumerate}
			\item[(ix,a)] there exists a bornologous map $g\colon Y\to X$, called {\em coarse inverse of $f$} such that $g\circ f\sim id_X$ and $f\circ g\sim id_Y$;
			\item[(ix,b)] $f$ is effectively proper and 
large-scale surjective.
		\end{enumerate}
	\end{enumerate}
	A family of maps $\{f_i\colon X\to Y\}_{i\in I}$ is {\em uniformly bornologous} if, for every $E\in\E_X$, there exists $F\in\E_Y$ such that $(f_i\times f_i)(E)\subseteq F$, for every $i\in I$. 
\end{definition}

\begin{proposition}[%
{see \cite[Proposition 2.7]{Za}}%
]\label{coro:ls_bijective_ce}
	Let $f\colon(X,\E_X)\to(Y,\E_Y)$ be a map between coarse spaces. Then $f$ is a coarse equivalence if and only if the following conditions hold:
	\begin{enumerate}[(i)]
		\item $f$ is both large-scale injective and large-scale surjective;
		\item $f$ is bornologous;
		\item $f$ is weakly uniformly bounded copreserving.
	\end{enumerate}
\end{proposition}

\subsection{Coarse groups}\label{sub:cg}

If $G$ is a group and $g\in G$, we define the {\em left-shift $s_g^\lambda\colon G\to G$} and the {\em right-shift $s_g^\rho\colon G\to G$} as follows: for every $x\in G$, $s_g^\lambda(x)=gx$ and $s_g^\rho(x)=xg$. The following property of left-shifts is easy to check: 

\begin{proposition}\label{prop:compatible*}
	Let $G$ be a group and $\E$ be a coarse structure on it. Then the following properties are equivalent:
	\begin{enumerate}[(i)]
		\item for every $E\in\E$, $GE=\{(gx,gy)\mid g\in G,\,(x,y)\in E\}\in\E$;
		\item the family $\mathcal S_G^\lambda=\{s_g^\lambda\mid g\in G\}$ is uniformly bornologous, i.e., for every $E\in\E$, there exists $F\in\E$ such that, for every $g\in G$, $(s_g^\lambda\times s_g^\lambda)(E)\subseteq F$.
	\end{enumerate}       
\end{proposition}

\begin{definition}
	A coarse structure $\E$ on a group $G$ is said to be a {\em left group coarse structure}%
\footnote{
It is also said to be \emph{left-invariant} in \cite{rosendal} and \emph{compatible} in \cite{nicas_rosenthal}.}
if it has the equivalent properties from Proposition \ref{prop:compatible*}. 
	A {\em left coarse group} is a pair $(G, \E)$ of a group $G$ and a left group coarse structure $\E$ on $G$. {\em Right group coarse structure} and {\em right coarse group} can be defined analogously. 
\end{definition}

The unit element of a group $G$ is denoted by $e_G$.
In order to define our leading example of left/right group coarse structures and left/right coarse groups (we shall see below that these are all possible coarse group structures and coarse groups) we need the following fundamental concept. 

\begin{definition}\label{def:group_ideal}
	Let $G$ be a group. A {\em group ideal} $\II$ (\cite{protasov_protasova}) is a family of subsets of $G$ containing the singleton $\{
e_G
\}$ such that:
	\begin{enumerate}[(i)]
		\item $\II$ is an ideal;
		\item for every $K,J\in\II$, $K\cdot J=\{kj\mid k\in K,\,j\in J\}\in\II$;
		\item for every $K\in\II$, $K^{-1}=\{k^{-1}\mid k\in K\}\in\II$.
	\end{enumerate}
\end{definition}
If $\II$ is a group ideal on $G$, 
then
$\bigcup\II$ is a subgroup of $G$. 

\begin{definition}\label{def:ideal}
	Let $G$ be a group and $\II$ be a group ideal. For every $K\in\II$, we define
	$$
	E_K^\lambda
=\{(x,y) \in G\times G \mid x^{-1} y \in K \}
=G(
\{
e_G
\}\times K)=\bigcup_{g\in G}(\{g\}\times gK).
	$$
	The family $\E_\II^\lambda=\{E\subseteq G\times G\mid\exists K\in\II:\,E\subseteq E_K^\lambda\}$ is a left coarse group structure, called {\em left $\II$-group coarse structure}, and the pair $(G,\E_{\II}^\lambda)$ is a left coarse group, called {\em left $\II$-coarse group}. 
\end{definition}
Note that the family $\{E_K^\lambda\mid K\in\II\}$ is a base of the $\II$-group coarse structure. Moreover, for every $K\in\II$ and $x\in G$, $E_K^\lambda[x]=xK$.

Similarly, we can define the {\em right $\II$-group coarse structure} $\E_\II^\rho$ and then the notion of {\em right $\II$-coarse group}. For every group $G$ and group ideal $\II$ on it, the left $\II$-group coarse structure and the right $\II$-group coarse structure are equivalent, as the following result shows.

\begin{proposition}\label{prop:inv_asym}
	Let $G$ be a group, $\II$ be a group ideal, and $\iota\colon G\to G$ such that $\iota(g)=g^{-1}$. Then $\iota\colon(G,\E_\II^\lambda)\to(G,\E_{\II}^\rho)$ is a bijective coarse equivalence. 
\end{proposition}

The following fact from \cite[%
Propositions 2.4 and 2.5%
]{nicas_rosenthal} shows that every left coarse group can be obtained as in Definition \ref{def:ideal} above.

\begin{proposition}\label{prop:compatible}
	Let $G$ be a group and $\E$ be a coarse structure on it. Then the following properties are equivalent:
	\begin{enumerate}[(i)] 
		\item $(G,\E)$ is a left coarse group; 
		\item $\E=\E_{\II}^\lambda$, where $\II=\{E[
e_G
]\mid E\in\E\}$.
	\end{enumerate}       
\end{proposition}
A similar result can be stated for right coarse structures.

Justified by Propositions \ref{prop:inv_asym} and \ref{prop:compatible}, in the sequel we will always refer to left group coarse structures and left coarse groups, if it is not otherwise stated, and thus we call them briefly group coarse structures (and coarse groups) if there is no risk of ambiguity. Moreover, if $G$ is a group, and $\II$ is a group ideal on it, we simply 
write
$(G,\II)$ to denote the associated coarse group
$(G, \E_{\II}^\lambda)$.

According to Proposition \ref{prop:compatible}, coarse groups are equivalently described in terms of group ideals. This is why it is necessary to provide examples of group ideals.

\begin{example}\label{ex:group_c_s}  Let $G$ be a group.
\begin{enumerate}[(i)]
\item The singleton $\{\{e_G\}\}$ is a group ideal (more precisely, the group ideal is $\{\{e_G\},\emptyset\}$, however, for the sake of brevity, we simplify the notation) and the $\{\{e_G\}\}$-group coarse structure is the one that contains only the subsets of the diagonal.
\item On the opposite side we have the group ideal $\mathcal P(G)$, that induces the coarse structure in which every subset of $G\times G$ is an entourage.
\item The family $[G]^{<\omega}$ of all finite subsets of $G$ is a group ideal and the $[G]^{<\omega}$-coarse structure is finitary, i.e., $\sup\{\lvert E[x]\rvert \mid x \in X \}$ is finite for any entourage $E$.
\end{enumerate} 

\end{example}
More examples of group ideals will be provided in Section \ref{sec:group_ideals_topgrp}.

\subsection{Description of large scale properties by group ideals}\label{sec:group_ideals}

Group ideals are very useful to characterise large-scale properties of spaces or maps. 

The first example regards connectedness. For a coarse group $(G,\II)$, the following properties are equivalent:
\begin{enumerate}[(i)]
\item $(G,\II)$ is connected;
\item $\bigcup\II=G$;
\item $[G]^{<\omega}\subseteq\II$.
\end{enumerate}

Let $f,g\colon X\to(G,\II)$ be two maps from a set to a coarse group. Then $f$ and $g$ are close if and only if there exists $M\in\II$ such that, for every $x\in X$, $(f(x),g(x))\in 
E_M^\lambda
$ or, equivalently, $g(x)\in f(x)M$. 

One can obtain useful characterisations of morphism properties in terms of group ideals.

\begin{proposition}[%
{\cite[Proposition 2.5]{nicas_rosenthal}, \cite[Proposition 2.2]{dikranjan_zava_qhomo}}%
]
\label{prop:homomorphism_group_ideal}
	Let $G$ and $H$ be two groups and $f\colon G\to H$ be a homomorphism between them. Let $\II_G$ and $\II_H$ be two group ideals on $G$ and $H$ respectively. Then:
	\begin{enumerate}[(i)]
		\item $f\colon(G,\II_G)\to(H,\II_H)$ 
is large-scale injective
if and only if $\ker f\in\II_G$;
		\item $f\colon(G,\II_G)\to(H,\II_H)$ is bornologous if and only if $f(I)\in\II_H$, for every $I\in\II_G$;
		\item $f\colon(G,\II_G)\to(H,\II_H)$ is effectively proper if and only if $f^{-1}(J)\in\II_G$, for every $J\in\II_H$;
		\item $f\colon(G,\II_G)\to(H,\II_H)$ is weakly uniformly bounded copreserving if and only if, for every $J\in\II_H$, there exists $I\in\II_G$ such that $J\cap f(G)\subseteq f(I)$.
	\end{enumerate}
\end{proposition}

By Propositions \ref{coro:ls_bijective_ce} and \ref{prop:homomorphism_group_ideal}, coarse equivalences can be characterised in terms of groups ideals as follows.

\begin{corollary}\label{coro:homo_ce}
	Let $f\colon (G,\II_G)\to(H,\II_H)$ be a homomorphism between coarse groups. Then $f$ is a coarse equivalence if and only if the following conditions hold:
	\begin{enumerate}[(i)]
		\item $\ker f\in\II_G$;
		\item there exists $J\in\II_H$ such that $f(G)\cdot J=H$;
		\item for every $I\in\II_G$, $f(I)\in\II_H$;
		\item for every $J\in\II_H$, there exists $I\in\II_G$ such that $J\cap f(G)\subseteq f(I)$.
	\end{enumerate}
\end{corollary}

\section{Group ideals on topological groups}\label{sec:group_ideals_topgrp}

Let $\TopGrp$ denote the category of topological group and countinuous homomorphisms between them.

In this section we introduce and discuss some group ideals on topological groups that agree with the topological structure. To formalise this idea, we introduce the following notion. 
\begin{definition}\label{def:functorial_group_ideal}
Let $\AA$ associate to every topological group $G$ a 
family $\AA(G)$ of subsets of $G$. 
Then $\AA$ is called a 
{\em group ideal} 
	if, 
for every topological group $G$, $\AA(G)$ is a group ideal. If, moreover, 
for every morphism of $\TopGrp$ $f\colon G\to H$, $f\colon(G,\AA(G))\to(H,\AA(H))$ is bornologous, or, equivalently, $f(\AA(G))\subseteq\AA(H)$
, then $\AA$ is said to be {\em functorial}.
\end{definition}
In \cite{dikranjan_zava_qhomo,dikranjan_zava_Pon} this notion has been introduced, providing the precise categorical setting, and studied. 
The groups ideals in Example \ref{ex:group_c_s} are functorial.
Note that to prove the functoriality of those group ideals the continuity of the homomorphism is not needed, and, in the notation of \cite{dikranjan_zava_qhomo}, those are called functorial group ideals on the category $\Grp$ of abstract groups.

We now want to introduce the two fundamental functorial group ideals mentioned in the introduction. Let us first recall the definition of a pseudonorm. Given a group $G$, a {\em pseudonorm} is a map $f\colon G\to\R_{\geq 0}$ satisfying the following properties:
\begin{enumerate}[(i)]
	\item $f(
e_G
)=0$;
	\item $f(g)=f(g^{-1})$, for every $g\in G$ ({\em symmetry});
	\item $f(gh)\leq f(g)+f(h)$, for every $g,h\in G$ ({\em subadditivity}).
\end{enumerate}

A pseudonorm $f:G \to \mathbb{R}_{\geq 0}$ is called a \emph{norm} if $g=e_G$ whenever $f(g)=0$. 

Let us recall that a pseudometric $d$ on a group $G$ is said to be {\em left-invariant} if $d(g,h)=d(kg,kh)$ for every $g,h,k\in G$.
\begin{remark}
\label{remark:d_f}
There is a close connection between (pseudo)norms and left-invariant (pseudo)metrics on a group. Let $G$ be a topological group. Given a (pseudo)norm $f$ on $G$, 
the function
$d_f\colon G\times G\to\R_{\geq 0}$ defined by 
\begin{revty}
\end{revty} 
$d_f(x,y)=f(x^{-1}y)$ for every $x,y\in G$ is a left-invariant (pseudo)metric, while every left-invariant (pseudo)metric $d$ induces a (pseudo)norm 
$d_{e_G}$
such that 
$d_{e_G}(x)=d(e_G,x)$ 
for every $x\in G$. Moreover, 
\begin{itemize}
\item $d_f$ is continuous provided that $f$ is continuous; 
\item 
$d_{e_G}$
is continuous if $d$ is continuous;
\item $\diam_{d_f}(A) < \infty$ if and only if $\diam (f(A)) <\infty$ for any $A\subseteq G$;
\item $\diam_{d} (A) <\infty$ if and only if $\diam (d_{e_G} (A)) < \infty$ for any $A\subseteq G$. 
\end{itemize} 
\end{remark}

\begin{definition}\label{def:K_and_OB}
Let $G$ be a topological group.
\begin{enumerate}[(i)]
	\item The family $\mathcal K(G)$ consists of all relatively compact subsets of $G$ (i.e., those $A\subseteq G$ having compact closure).
	\item {\rm (\cite[Proposition 
2.15%
]{rosendal})} The family $\OB(G)$ consists of all the subsets $A$ of $G$ satisfying one of the following equivalent properties:
	\begin{enumerate}
		\item[(ii$_1$)] 
		for every continuous left-invariant pseudometric $d$ on $G$, $\diam_d(A) < \infty$;
		\item[(ii$_2$)] for every continuous pseudonorm $f\colon G\to\R_{\geq 0}$, $\diam (f(A))<\infty$;
		\item[(ii$_3$)] for every continuous action $\alpha$ of $G$ on a metric space $X$ by isometries and every $x\in X$, $\diam(\alpha(A)x)<\infty$;
		\item[(ii$_4$)]for every increasing chain $V_1\subseteq\cdots\subseteq V_n\subseteq \cdots$ 
		of open subsets of $G$
		such that $V_i\cdot V_i\subseteq V_{i+1}$, for every $i\in\mathbb N$, and $\bigcup_nV_n=G$, there exists $k\in\mathbb N$ satisfying $A\subseteq V_k$;
		\end{enumerate}
\end{enumerate}
\end{definition}

Moreover, in \cite[%
Section 2.4%
]{rosendal} two other families of subsets of topological groups were introduced.
For a topological group $G$, let $\vartheta_\tau(e_G)$ denote the family of all neighbourhoods of the unit element $e_G$.
For $V \subseteq G$ and $n \in \mathbb{N}$, let $V^n =\{x_1x_2\cdots x_n \mid x_1, x_2, \dots, x_n \in V \}$.

\begin{definition}\label{def:group_ideals_V_and_F}
Let $G$ be a topological group. Let us define the following subsets of $G$:
\begin{gather*}
\VV
(G)
=\{A\subseteq G\mid\forall V\in\vartheta_\tau(
e_G
),\exists n\in\mathbb N:A\subseteq V^n\},\quad\text{and}\\
\FF
(G)
=\{A\subseteq G\mid\forall V\in\vartheta_\tau(
e_G
),\exists n\in\mathbb N,
\exists
F\in[G]^{<\omega}:A\subseteq (
FV
)^n\}.
\end{gather*}
\end{definition}
Note that, if a topological group $G$ has an open subgroup $V$, then $\bigcup\VV(G)\subseteq V$. In particular, if a group $G$ is discrete, then $\{e_G\}$ is an open subgroup, and thus $\VV(G)=\{\{e_G\}\}$.

For every topological group $G$, there are inclusions between the four families provided in Definitions \ref{def:K_and_OB} and \ref{def:group_ideals_V_and_F}:
\begin{equation}\label{eq:rosendal_families_inclusions}
\xymatrix@=7pt{
	& &\VV(G) \\
	\OB(G)\quad\ar@{}[r]|{\supseteq} &\FF(G)\ar@{}[dr]|{\leqdr}\ar@{}[ur]|{\lequr} & \\
	& &\KK(G).
}
\end{equation}
The inclusions \eqref{eq:rosendal_families_inclusions} easily follows from the definition (see also \cite[%
p.34%
]{rosendal}).

\begin{theorem}
$\KK
$, $\OB
$, $\VV
$ and $\FF
$
 are functorial group ideals
 .
\end{theorem}
\begin{proof}
The fact that $\VV(G)$, $\KK(G)$, $\FF(G)$, and $\OB(G)$ are group ideals 
for every topological group $G$ 
 is known (see, for example, \cite[%
 Example 2.10%
 ]{nicas_rosenthal} for $\KK(G)$, while \cite[%
 Section 2.4%
 ]{rosendal} for the others). 
It is easy to see that $\KK$ is functorial.
In \cite[Lemma 2.46]{rosendal}, it was proved that $\OB$ is functorial.

Let us show that 
$\VV
$ 
 and 
 $\FF
$ 
 are functorial. Fix a continuous homomorphism $f\colon G\to H$ between two topological groups and $A\subseteq G$. Suppose that $A\in\VV(G)$, and let $V\in\vartheta_\tau(e_H)$. Since $f$ is a continuous homomorphism, $f^{-1}(V)\in\vartheta_\tau(e_G)$ and thus there exists $n\in\mathbb N$ satisfying $A\subseteq (f^{-1}(V))^n$. Then $f(A)\subseteq f((f^{-1}(V))^n)\subseteq V^n$, which shows that $f(A)\in\VV(H)$. Now assume that $A\in\FF(G)$ and again fix $V\in\vartheta_\tau(e_H)$. Similarly, there exist $n\in\mathbb N$ and $F\in[G]^{<\omega}$ such that $A\subseteq (f^{-1}(V)F)^n$. Then $f(A)\subseteq (Vf(F))^n$, which leads to the conclusion since $f(F)$ is trivially finite.
\end{proof}

\begin{remark}[\cite{rosendal}]
\label{remark:left-coar-st-left-unif}
%
%
Let us recall that, according to Definition \ref{def:K_and_OB}(ii$_2$), for every topological group $G$,
$$\E_{\OB(G)}^\lambda=\bigcap\{\E_d\mid \text{$d
$ is a continuous left-invariant pseudometric
on $G$%
}\},$$
which is also called the {\em left-coarse structure} and denoted by $\E_L$. This rewriting makes evident the strong parallelism with the dual notion of the {\em left-uniformity $\mathcal U_L$} of a topological group $G$, which is defined by:
$$\mathcal U_L=\bigcup\{\mathcal U_d\mid\text{$d
$ is a continuous left-invariant pseudometric
on $G$%
}\},$$
where $\mathcal U_d$ denotes the metric uniformity induced by the pseudometric $d$.
\end{remark}

Inspired by Definition \ref{def:K_and_OB}(ii$_2$,ii$_3$), we introduce more functorial group ideals induced by families of maps (\S\ref{sub:functorial_gi_from_maps}) or by actions (\S\ref{sub:functioral_gi_by_actions}). 

A metric on a group will always be considered left-invariant.
%

\subsection{Group ideals induced by families of maps}\label{sub:functorial_gi_from_maps}

\begin{definition}
Let $G$ be a group and $(H,d)$ be a metric group. A map $f\colon G\to H$ is said to be a {\em quasi-homomorphism} if there exists $C\in\R_{\geq 0}$ such that, for every $x,y\in G$, $d(f(xy),f(x)f(y))\leq C$. 
\end{definition}
The notion of a quasi-homomorphism dates back to Ulam (\cite[Section 6.1]{ulam}), and it has been studied in geometric group theory (see, for example, \cite{calegari},  \cite{fujiwara_kapovich}, \cite{kotschick}). Moreover,
it has been generalised by Rosendal to maps from groups to coarse groups (\cite{rosendal}). This generalisation is central in the construction of some categories of coarse groups proposed in \cite{dikranjan_zava_qhomo}.

Note that, if $f\colon G\to(H,d)$ is a quasi-homomorphism between a topological group $G$ and a metric group $H$, and $C\geq 0$ such that $d(f(xy),f(x)f(y))\leq C$, for every $x,y\in G$, then
$$d(f(
e_G
),
e_H
)=d(f(
e_G
)^2,f(
e_G
\cdot 
e_G
))\leq C,$$
and moreover, for every $x\in G$,
$$d(f(x^{-1}),f(x)^{-1})=d(
f(x)f(x^{-1})
,e_H)\leq d(
f(x)f(x^{-1})
,f(x^{-1}x))+d(f(
e_G
),
e_H
)\leq 2C.$$
Hence, without loss of generality, for a suitable constant $C\geq 0$, we can assume that the following properties hold:
\begin{enumerate}[(i)]
\item $d(f(e_G),e_H)\leq C$;
\item for every $x,y\in G$, $d(f(xy),f(x)f(y))\leq C$;
\item for every $x\in G$, $d(f(x^{-1}),f(x)^{-1})\leq C$.
\end{enumerate}
If (i)--(iii) are satisfied, we say that $f$ is a {\em $C$-quasi-homomorphism}.

Let $G$ be a topological group and $H$ be a metric group. Let us define the following families of maps:
\begin{gather*}
\MM_{\BB}(G,H)=\{f\colon G\to H\mid \mbox{$f$ is continuous}\},\\
\MM_{\QH}(G,H)=\{f\colon G\to H\mid \mbox{$f$ is a continuous quasi-homomorphism}\},
\quad\mbox{and}\\
\MM_{\HH}(G,H)=\{f\colon G\to H\mid \mbox{$f$ is a continuous homomorphism}\}.
\end{gather*}

Then the following inclusions hold:

\begin{equation}\label{eq:morphism_families_inclusions}
\MM_{\BB}(G,H) \supseteq \MM_{\QH}(G,H) \supseteq \MM_{\HH}(G,H). 
\end{equation}

\begin{definition}\label{def:group_ideals_from_maps}
Let $\XX\in\{\BB,
\QH,
\HH\}$. For every topological group $G$ and every metric group $(H,d)$, define
\begin{gather*}\XX(G,H)=\bigcap_{f\in\MM_{\XX}(G,H)}\{A\subseteq G\mid\diam(f(A))<\infty\},\\
\XX(G)=\bigcap\{\XX(G,H)\mid \mbox{$H$ is a metric group}\},\quad
\mbox{and}\quad\XX_{\R}(G)=\XX(G,\R).
\end{gather*}
\end{definition}
%

Trivially, the following inclusions hold for every topological group $G$:
\begin{equation}\label{eq:morphisms_group_ideals_relationships}
\xymatrix@=7pt{
\BB_\R(G)\ar@{}[d]|{\leqd} \ar@{}[r]|{\subseteq}& \,\,\,\,\QH_\R(G)\,\,\,\ar@{}[r]|{\subseteq} \ar@{}[d]|{\leqd}& \HH_\R(G)\ar@{}[d]|{\leqd} \\
\BB(G)\ar@{}[r]|{\subseteq}&\,\,\,\,\QH(G)\,\,\,\ar@{}[r]|{\subseteq} & \HH(G). }
\end{equation}

\begin{proposition}\label{prop:R_is_enough}
	Let $G$ be a topological 
group.
Then the following properties hold:
	\begin{enumerate}[(i)]
		\item $\BB(G)=\BB_{\R}(G)$;
		\item 
$\BB(G) \subseteq \OB(G) \subseteq  \QH(G)$.
	\end{enumerate}	
\end{proposition}
\begin{proof}
	(i) Of course, $\BB(G)\subseteq\BB_\R(G)$. Let now $A\in\BB_\R(G)$ and take a continuous map $f\colon G\to(H,d)$, where $d$ is a left-invariant metric on $H$. 
Let $d_{e_H}\colon H\to\R_{\geq 0}$ be the pseudonorm defined by $d_{e_H}(x) =d(e_H,x)$ for $x \in H$ (Remark \ref{remark:d_f}).
Since the map $
d_{e_H}
\circ f\colon G\to\R$ is continuous, 
$\diam(d_{e_H}(f(A)))<\infty$ 
which implies that $\diam(f(A))<\infty$. 
	
	(ii) 
The inclusion $\BB(G) \subseteq \OB (G)$ is obvious (see the condition (ii$_2$) in Definition \ref{def:K_and_OB}).
Let us show that  $\OB(G)\subseteq \QH(G)$.
The following proof is similar to that of the implication (4)$\Rightarrow$(2) in \cite[Proposition 
2.15%
]{rosendal}. Fix an element $A\in\OB(G)$. Let now $f\colon G\to H$ be a continuous 
$C$-quasi-homomorphism.
For $n \in \mathbb{N}$, let $V_n = \{g \in G\mid d(f(g),e_H)<(2+C)^n \}$. Then $G=\bigcup_{n\in \mathbb{N}}V_n$ and each $V_n$ is open in $G$ since $f$ is continuous. Moreover, $
V_n^2
\subseteq V_{n+1}$ for every $n\in \mathbb{N}$. Indeed, let $n \in \mathbb{N}$ and $g,h \in V_n$. Since $d$ is left-invariant, we have 
\begin{align*}
d(f(gh),e_H)&\,\leq  
d(f(gh),f(g)f(h))+ d(f(g)f(h),f(g)) + d(f(g),e_H)\leq\\
&\, \leq C+ d(f(g)f(h),f(g)) + d(f(g),e_H)=\\
&\,= C+ d(f(h),e_H) + d(f(g),e_H) <C+ 2(2+C)^n  <(2+C)^{n+1}, 
\end{align*}
which implies $gh \in V_{n+1}$, and thus  $V_n^2 \subseteq V_{n+1}$. 
	Since $A \in \OB(G)$, there is $ n\in \mathbb{N}$ such that $A \subseteq V_n$. Then we have $\diam_d (f(A)) \leq 2(2+C)^n$, and so $f(A)$ is $d$-bounded.
\end{proof}

By \eqref{eq:morphisms_group_ideals_relationships} and Proposition \ref{prop:R_is_enough},
the following inclusions hold for every topological group $G$:
\begin{equation}\label{eq:morphisms_group_ideals_relationships-2}
\xymatrix@=7pt{
& & & \QH_\R(G)\ar@{}[dr]|{\geqdr} & \\
\BB(G)\ar@{}[r]|{\subseteq} &\,\, \OB(G)\ar@{}[r]|{\subseteq}&\,\,\QH(G)\ar@{}[dr]|{\geqdr} \ar@{}[ur]|{\gequr}& & \HH_\R(G). \\
& & & \HH(G)\ar@{}[ur]|{\gequr} &
}
\end{equation}

\begin{theorem}
$\BB$, $\QH$, $\HH$, $\QH_{\R}$ and $\HH_{\R}$ are functorial group ideals.
\end{theorem}
\begin{proof}
	Let $G$ be a topological group.
	
	We want to show first that $\BB(G)$ is a group ideal. The only non-trivial property is that the family $\BB(G)$ is closed for finite products of their  
elements, which follows from \cite[Corollary 2.15]{tkachenko}.

	Let us prove that $\QH(G)$ is a group ideal. Let $A,B\in\QH(G)$, and $f\colon G\to H$ be a $C$-quasi-homomorphism, where $(H,d)$ is a metric group, for $C\geq 0$. We can assume without loss of generality that $e_G\in A\cap B$. Pick an element $a\in A$ and $b\in B$. Then 
	\begin{align*}
	d(f(ab),e_H)&\,\leq d(f(ab),f(a)f(b))+d(f(a)f(b),f(a))+d(f(a),
e_H
)\leq\\
	&\,\leq C+d(f(b),
e_H
)+d(f(a),
e_H
)\leq\\
&\,\leq C+d(f(b),f(
e_G
))+d(f(a),f(
e_G
))+2d(f(
e_G
),
e_H
)\leq\\
	&\,\leq 3C+\diam (f(A))+\diam (f(B)),
	\end{align*}
	which implies that $\diam (f(AB))\leq 2(3C+\diam (f(A))+\diam (f(B)))<\infty$. Moreover, for every $a\in A$,
	\begin{align*}d(f(a^{-1}),
e_H
)&\,\leq d(f(a^{-1}),f(a^{-1})f(a))+d(f(a^{-1})f(a),f(
e_G
))+d(f(
e_G
),
e_H
)\leq\\
	&\,\leq d(
e_H
,f(
e_G
))+d(f(
e_G
),f(a))+d(f(a^{-1})f(a),f(
e_G
))+d(f(
e_G
),
e_H
)\leq\\
&\,\leq \diam (f(A))+3C,\end{align*}
	and thus $\diam (f(A^{-1}))\leq 2(\diam (f(A))+3C)<\infty$. The proofs that $\QH_{\R}(G)$, $\HH(G)$ and $\HH_{\R}(G)$ are group ideals are similar.
	
To prove that 
$\BB$, $\QH$, $\QH_{\R}$, $\HH$, and $\HH_{\R}$ 
are functorial let us first note that, if $f\colon G\to H$ is a continuous homomorphism between two topological groups, then, for every $g\colon H\to K$ continuous map (quasi-homomorphism, or homomorphism) to a metric group, 
$g\circ f$ is continuous (a quasi-homomorphism, or homomorphism, respectively). Then the functoriality can be easily derived.
\end{proof}

To conclude this subsection, we provide two remarks discussing other potential families of subsets of a topological group induced by maps.	
	
\begin{remark}
We mentioned that Definition \ref{def:K_and_OB}(ii$_2$) was the inspiration for the introduction of the group ideals discussed in this subsection. However, the reader may wonder why we treated the group ideal $\OB(G)$ differently from the others as there is no group ideal induced by the family of continuous ``pseudonorms'' having image into an arbitrary metric group. The present remark provides an explanation. 

Let $G$ be a group and $(H,d)$ be a metric group. We
might
define a {\em generalised pseudonorm of $G$ with values in $H$} as a map $f\colon G\to H$ satisfying the following requests:
\begin{enumerate}[(i)]
	\item $f(e_G)=e_H$;
	\item for every $g\in G$, $d(f(g),e_H)=d(f(g^{-1}),e_H)$ ({\em symmetry});
	\item for every $g,h\in G$, $d(e_H,f(gh))\leq d(e_H,f(g)f(h))$ ({\em subadditivity}).
\end{enumerate}
%
%
%
However, this notion is essentially the same as a pseudonorm in the following sense:
Every pseudonorm of $G$ is a generalised pseudonorm of $G$ with values in $\R$.
Conversely, let $f: G \to H$ be a generalised pseudonorm to a metric group $(H,d)$.
Then the map $f': G \to \mathbb{R}$ defined by $f'(g)=d(e_H,f(g))$ for $g \in G$ is a pseudonorm such that $\diam (f'(A)) < \infty$ if and only if $\diam_d(f(A))< \infty$ for any $A \subseteq G$.
Thus the set of all $A \subseteq G$ such that $\diam_d(f(A)) < \infty$ for any continuous generalized pseudonorm $f: G \to H$ to a metric group $(H,d)$ coincides with $\OB(G)$.
\end{remark}

\begin{remark}\label{rem:UB}
	Let, for every topological group $G$ and metric group $H$, $\MM_{\UB}(G,H)$ be the family of all uniformly continuous maps from $G$ to $H$. 
	Then, we can define the families $\UB(G)$ and $\UB_{\R}(G)$ as in Definition \ref{def:group_ideals_from_maps}. 
	Similarly to Proposition \ref{prop:R_is_enough}(i), we can show that $\UB(G)=\UB_{\R}(G)$, for every topological group $G$. Moreover, according to \cite[Example 
	2.27%
	]{rosendal},
	$$\UB(G)=\UB_{\R}(G)=\{A\subseteq G\mid\forall V\in\vartheta_\tau(
e_G
),\exists F\in[G]^{<\omega},
\exists
n\in\mathbb N:A\subseteq FV^n\}.$$
	
	However, the family $\mathcal{UB}(G)$ is not a group ideal in general. 
	Consider the symmetric group 
	$S_\infty$
	of all permutations of $\mathbb{N}$ 
equipped with the Polish topology obtained by declaring point-wise stabilisers of finite subsets open.
For $k \in \mathbb{N}$, let $\sigma_k\colon \mathbb{N} \to \mathbb{N}$ be the cycle of the first $k$ letters defined by letting 
$$\sigma_k(n)  =\begin{cases}
n+1  & \text{if } n < k,\\
1 & \text{if } n=k,\\ 
n & \text{if }n> k.
\end{cases}$$
Set $A=\{\sigma_k\mid k\in\mathbb{N}\}$	
(see \cite[Example 2.27]{rosendal}).
Then 
$A\subseteq S_\infty$.

To show $A \in \mathcal{UB}(S_\infty)$, let $V\in\vartheta_\tau(e)$, where $e$ is the identity map on $\mathbb{N}$. Then there exists $m \in \mathbb{N}$ such that 
$\{ \sigma \in S_\infty \mid  \forall i<m : \sigma(i) =i\} \subseteq V$.
Let $F=\{ \sigma_i \mid i \leq m  \}$.
Then we have $A \subseteq FV$, and hence $A \in \mathcal{UB}(S_\infty)$.  In fact, if $k>m$, 
then
$\sigma_k=\sigma_m\circ \rho$, where $\rho\in V$ is defined as follows: for every $n\in \mathbb N$,
$$
\rho(n)=\begin{cases}
n & \text{ if $n<m$ or $n>k$,}\\
n+1 & \text{ if $m\leq n<k$,}\\
m & \text{ if $n=k$.}
\end{cases}$$
	
To show $A^{-1} \notin \mathcal{UB}(S_\infty)$, 
let $V= \{\sigma \in S_\infty  \mid \sigma(1) =1 \}$.
Then $V\in\vartheta_\tau(e)$.
Let $F$ be an arbitrary finite subset of $S_\infty$ and $n \in \mathbb{N}$.
Then $FV^n =FV$. Thus it suffices to show that $A^{-1} \not\subseteq FV$.
Let $m = \max \{\tau (1) \mid \tau \in F\} +1$.
Then $\sigma_m^{-1} \in A^{-1}$.
Since $\sigma_m(m) = 1$, we have $\sigma_m^{-1} (1) =m \notin \{\tau(1) \mid\tau  \in F\}$, which implies 
$\sigma_m^{-1}  \notin FV$. 
Therefore $A^{-1} \not\subseteq FV$. 
\end{remark}

\subsection{Group ideals induced by families of actions}\label{sub:functioral_gi_by_actions}

Let us define some more families of subsets of topological groups. Let $G$ be a topological group and $X$ be a metric space. Denote by $\Act(G,X)$ the family of all continuous isometric actions of $G$ on $X$, and, for every action $\alpha\in\Act(G,X)$, 
$$
\OrB 
(G,\alpha)=\{A\subseteq G\mid\forall x\in X,\diam(\alpha(A)x)<\infty\}.$$
Note that, by definition, for a topological group $G$, 
$$\OB(G)=\bigcap_{\text{$X$ metric space}}\bigg(\bigcap_{\alpha\in\Act(G,X)}
\OrB 
(G,\alpha)\bigg).$$
In this work, we also consider the following families.
\begin{definition}\label{def:group_ideals_from_actions}
For a topological group $G$, we define
\begin{gather*}
\UC(G)=\bigcap_{\text{$X$ ultra-metric space}}\bigg(\bigcap_{\alpha\in\Act(G,X)}
\OrB 
(G,\alpha)\bigg),\quad\mbox{and}\\
\FH(G)=\bigcap_{\text{$H$ Hilbert space}}\bigg(\bigcap_{\text{$\alpha\in\Act(G,H)$ acting via affine maps}}
\OrB 
(G,\alpha)\bigg).
\end{gather*}
\end{definition}
Clearly, for every topological group $G$, 
\begin{equation}\label{eq:action_group_ideals_relationships}
\UC(G)\supseteq\OB(G)\subseteq\FH(G).
\end{equation} 

\begin{remark}
	\label{rem:prop.FH}
	A topological group $G$ is said to have \emph{property (FH)} if every continuous affine isometric action of G on a real Hilbert space has a global fixed point.
	Note that a topological group $G$ has property (FH) if and only if $G \in \mathcal{FH}(G)$ (see \cite[Proposition 2.2.9]{bekka_harpe_valette}).
	
	For example, the group $SL_3(\mathbb Z)$ with the discrete topology satisfies $SL_3(\mathbb Z) \in \mathcal{FH}(SL_3(\mathbb Z))$ (see \cite[Example 1.7.4 and Theorem 2.12.4]{bekka_harpe_valette}).
\end{remark}

\begin{proposition}
\label{prop:UC_pseudometric}
Let $G$ be a topological 
group.
Then a subset $A$ of $G$ belongs to $\UC(G)$ if and only if, for every continuous left-invariant ultra-pseudometric $d$ of $G$, $\diam_d(A)<\infty$.	
\end{proposition}
\begin{proof}
Let $A\subseteq G$ and suppose that $A\in\UC(G)$. Let $d$ be a continuous left-invariant ultra-pseudometric. Let $H=\overline{\{
e_G
\}}^{d}$ which is a subgroup of $G$ since $d$ is left-invariant. On the left coset space $G/H$ we induce a continuous left-invariant ultra-metric $\delta$ defined as follows: for every $g,h\in G$, $\delta(q(g),q(h))=d(g,h)$, where $q\colon G\to G/H$ is the projection map. Because of the definition of $H$, $\delta$ is well-defined and it is a metric. Moreover, it is left-invariant and continuous since it is composition of continuous maps. For every $g,h,k\in G$,
	$$\delta(q(g),q(h))=d(g,h)\leq\max\{d(g,k),d(k,h)\}=\max\{\delta(q(g),q(k)),\delta(q(k),q(h))\},$$
	and so $\delta$ is an ultra-metric. Moreover $\diam_d(A)=\diam_\delta(q(A))$, and $G$ acts on $G/H$ with the desired property, which implies that $\diam_d(A)<\infty$.
	
	Conversely, let $A\subseteq G$ be a subset of $G$ such that, for every continuous left-invariant ultra-pseudometric $d$ of $G$, $\diam_d(A)<\infty$. Let $G\curvearrowright X$ be a continuous action via isometry on a ultra-metric space $(X,d)$. Then we induce a continuous ultra-pseudometric on $G$ as follows: for a fixed point $x\in X$, define $\delta_x(g,h)=d(gx,hx)$, for every $g,h\in G$. Then it is easy to check that $\delta_x$ has the desired properties. Finally, $\diam_{\delta_x}(A)=\diam_d(Ax)<\infty$.
\end{proof}

\begin{theorem}
\label{thm:UH-FH-funct-gp-ideal}
$\UC$ and $\FH$ are functorial group ideals.
\end{theorem}
\begin{proof}
Let $G$ be a topological group. 
Let us prove 
that $\UC(G)$ is a group ideal. Let $A,B\in\UC(G)$ satisfying, without loss of generality, $e_G\in A\cap B$, and 
$\alpha$ be a continuous action of $G$ via isometries on an ultra-metric space $X$. 
Let $a\in A$ and $b\in B$. Then, for every $x\in X$,
\begin{align*}d(\alpha(ab)x,x)&\,\leq d(\alpha(a)\alpha(b)x,\alpha(a)x)+d(\alpha(a)x,x)=d(\alpha(b)x,x)+d(\alpha(a)x,x)\leq\\
&\,\leq\diam(\alpha(B)x)+\diam(\alpha(A)x)\end{align*}
since $G$ acts via isometries. Then $\diam(\alpha(AB)x)\leq 2(\diam(\alpha(A)x)+\diam(\alpha(B)x))$ and so $AB\in\UC(G)$. We can similarly show that also $\FH(G)$ is a group ideal.

The functoriality of 
$\UC$ and $\FH$ 
can be easily shown. Let $f\colon G\to H$ be a continuous homomorphism between two topological groups. Then, if $\alpha$ is a continuous action of $H$ by isometries on a metric space $X$ (with further properties), the action $\beta$ of $G$ on $X$ 
%
%
defined by $\beta(g)(x)=\alpha(f(g))(x)$, for every $g\in G$ and $x\in X$ is a continuous action via isometries on the same space. This is the key observation from which the claim easily follows.
\end{proof}

\section{Relationships between group ideals}\label{sec:relationships_with_no_further_conditions}

Functorial group ideals can be seen as functors from $\TopGrp$ to $\Set$, associating to every topological group $G$ the corresponding group ideal. These are functors because the group ideals are functorial. We can introduce a pointwisely-defined partial order $\preceq$ on the class of functorial group ideals. Namely, for every two functorial group ideals $\mathcal C,\mathcal D\colon\TopGrp\to\Set$, we write $\mathcal C\preceq\mathcal D$ if, for every $G\in\TopGrp$, $\mathcal C(G)\subseteq\mathcal D(G)$. We want to study this partial order on the family of functorial group ideals we have defined so far: 
	


$$\mathfrak X=\{\{\{e_{-}\}\},\mathcal P,[-]^{<\omega},\KK,\OB,\VV,\FF,\BB,\QH,\HH,\QH_{\R},\HH_{\R},\UC,\FH\}.$$

\begin{proposition}\label{prop:finite_preceq_all}
For every $\mathcal X\in\mathfrak{X}\setminus\{\{\{e_{-}\}\},\VV\}$, $[-]^{<\omega}\preceq\mathcal X$.
In particular, for every topological group $G$, 
the coarse group
$(G,\XX(G))$ is connected. 
\end{proposition}

\begin{proposition}\label{prop:FH_preceq_HR}
	$\FH\preceq\HH_{\R}$.
\end{proposition}
\begin{proof}
	Let
	$G$ be a topological group, 
	$A\in\FH(G)$ and $f\colon G\to\R$ be a continuous homomorphism. Define an action $\psi_f$ of $G$ on $\R$ as follows: for every $(g,x)\in G\times\R$, $\psi_f(g,x)=f(g)+x$. Then $\psi_f$ is a continuous isometric action via affine maps on the Hilbert space $\R$. Thus $\diam (f(A))=\diam 
(
f(A)+x
)
=\diam(
\psi_f
(A))<\infty$.
\end{proof}

\begin{proposition}\label{prop:K_preceq_B}
$\KK\preceq\BB
\preceq \FF
$.
\end{proposition}
\begin{proof}
It is easy to see that $\KK\preceq\BB$, and it is known that for any topological group $G$,  if  $B\in \BB(G)$ and $V\in\vartheta_\tau(e_G)$, then there exists $F \in [G]^{<\omega}$ such that $B \subseteq FV$ (see \cite[Proposition 6.10.2]{arhangelskii_tkachenko}), which implies $\BB \preceq \FF$.
\end{proof}

According to \eqref{eq:rosendal_families_inclusions}, 
\eqref{eq:morphisms_group_ideals_relationships-2},
\eqref{eq:action_group_ideals_relationships}, and Propositions \ref{prop:finite_preceq_all}, \ref{prop:FH_preceq_HR} and \ref{prop:K_preceq_B}, the following scheme represents the Hasse diagram of $(\mathfrak X,\preceq)$:
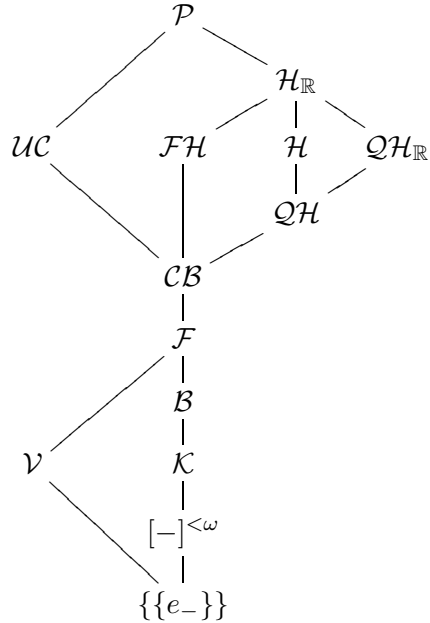
\begin{equation}\label{eq:Hasse_group_ideals}
\xymatrix@-1.2pc{& & \mathcal P
	\ar@{-}[dr]\ar@{-}[ddll] & &\\
& & &\HH_{\R}\ar@{-}[dl]\ar@{-}[dr]\ar@{-}[d] &\\
\UC\ar@{-}[ddrr] &  & \FH\ar@{-}[dd] &\HH\ar@{-}[d] &\QH_{\R}\ar@{-}[dl]\\
& &  & \QH\ar@{-}[dl] &\\
& & \OB\ar@{-}[d]& &\\
& &  \FF\ar@{-}[d]\ar@{-}[ddll] & &\\
& & \BB\ar@{-}[d] & &\\
\VV\ar@{-}[ddrr]&  & \KK\ar@{-}[d] & & \\
&  & [-]^{<\omega}
\ar@{-}[d] & &\\
& & \{\{e_{-}\}\}
& &}
\end{equation}
Let us warn the reader that \eqref{eq:Hasse_group_ideals} may not faithfully represent the Hasse diagram as we have not proved yet the non-existence of other relationships between the functorial group ideals of $\mathfrak X$. In Examples \ref{ex:FGAb}--\ref{ex:preceq} and 
Question \ref{q}
we address this topic.

In the sequel we want to understand also other partial orders on the class of functorial group ideals, closely related to $\preceq$. For a subclass $\mathbb G$ of $\TopGrp$, we define $\preceq_{\mathbb G}$ as the partial order defined by setting, for every pair of functorial group ideals $\mathcal C,\mathcal D\colon\TopGrp\to\Set$, $\mathcal C\preceq_{\mathbb G}\mathcal D$ if and only if $\mathcal C(G)\subseteq\mathcal D(G)$, for every $G\in\mathbb G$. In particular, $\preceq=\preceq_{\TopGrp}$. Moreover, we write $=_{\mathbb G}$ for the equivalence relation $\preceq_{\mathbb G}\cap\succeq_{\mathbb G}$.

We are interested in the partial orders induced by the following classes of topological groups:
\begin{enumerate}
\item[\FGAb:] the class of finitely generated discrete abelian groups;
\item[\FG:] the class of finitely generated discrete groups;
\item[$\countAb$:] the class of countable discrete abelian groups;
\item[$\countG$:] the class of countable discrete groups;
\item[\Ab:] the class of discrete abelian groups;
\item[\Grp:] the class of discrete groups;
\item[$\sigmacAb$:] the class of $\sigma$-compact locally compact abelian groups;
\item[$\sigmac$:] the class of $\sigma$-compact locally compact groups;
\item[\LCA:] the class of locally compact abelian groups;
\item[\LC:] the class of locally compact groups.
\end{enumerate}
Let us recall, as mentioned in the introduction, that every countable discrete topological group admits a left-invariant proper metric. Since it is proper, its bounded subsets are precisely the compact ones, i.e., the finite subsets.

For a better understanding of the relationship between these classes, let us represent the Hasse diagram of their containment:

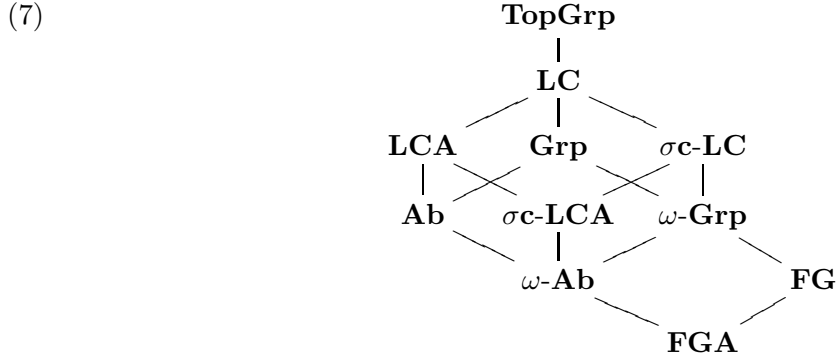
\begin{equation}\label{eq:classes_of_top_groups}
\xymatrix@-1.2pc{
& \TopGrp\ar@{-}[d] & & \\
& \LC\ar@{-}[dl]\ar@{-}[d]\ar@{-}[dr] & & \\
\LCA\ar@{-}[d]\ar@{-}[dr] & \Grp\ar@{-}[dl]\ar@{-}[dr] & \sigmac\ar@{-}[dl]\ar@{-}[d] & \\
\Ab\ar@{-}[dr] & \sigmacAb\ar@{-}[d] & \countG\ar@{-}[dl]\ar@{-}[dr] &\\
& \countAb\ar@{-}[dr] & & \FG\ar@{-}[dl]\\
& & \FGAb & 
}
\end{equation}

Let us state the following trivial, but useful result.
\begin{fact}\label{fact:induced_partial_order_reverted}
For every two subclasses $\mathbb A\subseteq\mathbb B\subseteq\TopGrp$, we have $\preceq_{\mathbb A}\supseteq\preceq_{\mathbb B}$.\end{fact} 
Hence, the Hasse diagram of the induced partial orders is the conjugated of \eqref{eq:classes_of_top_groups}.

We will discuss later in this paper whether the containment in \eqref{eq:Hasse_group_ideals} are proper or not, specialising the results in the classes of topological groups represented in \eqref{eq:classes_of_top_groups}.

\section{Stability properties of the classes $\XX\mbox{-}\YY$}\label{sec:stability}

Let $\XX$ and $\YY$ be two functorial group ideals satisfying $\XX\preceq\YY$. Define the class $\XX\mbox{-}\YY$ of those topological groups $G$ satisfying $\XX(G)=\YY(G)$. Hence, $\XX\mbox{-}\YY$ is the maximal subclass of $\TopGrp$ satisfying $\XX=_{\XX\mbox{-}\YY}\YY$.

A functorial group ideal $\XX$ is ({\em finitely}) {\em productive} if, for every 
(finite)
family of topological groups $\{G_i\}_{i\in I}$,
$$\XX(\Pi_iG_i)=\cl(\Pi_i\XX(G_i)),$$
with the notation that 
$\cl(\Pi_i\XX(G_i))= \{ A \subseteq \Pi_iG_i\mid \exists (A_i)_i \in \Pi_i\XX(G_i) : A \subseteq \Pi_i A_i\}$.

\begin{proposition}\label{prop:group_ideal_implies_prod}
Let $\XX$ be a functorial group ideal. Then $\XX$ is finitely productive.
\end{proposition}
\begin{proof}
Let $\{G_1,\dots,G_n\}$ be a family of topological groups and set $G=\Pi_iG_i$. For every $i\in\{1,\dots,n\}$, 
the canonical projection 
$p_i\colon G\to G_i$ is a continuous homomorphism, and thus, for every $A\in\XX(G)$, $p_i(A)\in\XX(G_i)$. Then $A\subseteq\Pi_ip_i(A)$, and so $\XX(G)\subseteq\cl(\Pi_i\XX(G_i))$. As for the opposite inclusion, let $B_i\in\XX(G_i)$, for every $i\in\{1,\dots,n\}$. Without loss of generality, we can assume that $e_i\in B_i$, where $e_i$ is the identity of $G_i$. Define $B=\Pi_iB_i$ and, for every $i\in\{1,\dots,n\}$, the section $s_i\colon G_i\to G$ of $p_i$ identifying $G_i$ with $\{e_1\}\times\cdots\times\{e_{i-1}\}\times G_i\times\{e_{i+1}\}\times\cdots\times\{e_n\}$. Then $s_i$ is a continuous homomorphism and thus $s_i(B_i)\in\XX(G)$. The conclusion follows since $B\subseteq s_1(B_1)\cdots s_n(B_n)\in\XX(G)$.
\end{proof}
It is worth mentioning that, for every family $\{G_i\}_{i\in I}$ of topological groups and every functorial group ideal $\XX$, 
\begin{equation}\label{eq:prod}\XX(G)\subseteq
\cl(\Pi_i\XX(G_i)).
\end{equation}
The proof is analogue to that 
of Proposition \ref{prop:group_ideal_implies_prod}.

\begin{proposition}\label{prop:arbitr_productive}
Both $\KK$ and $\OB$ are productive.
\end{proposition}
\begin{proof}
As for $\OB$, we refer to \cite[%
Lemma 3.36%
]{rosendal}. As for $\KK$, according to \eqref{eq:prod}, it is enough to show that, for every family $\{G_i\}_{i\in I}$ of topological groups and every family of subsets $A_i\in\KK(G_i)$, where $i\in I$, $\Pi_iA_i\in\KK(\Pi_iG_i)$. Since $A_i$ is relatively compact, then $\overline{A_i}$ is compact in $G_i$, and then $\Pi_i\overline{A_i}$ is compact and the conclusion follows.
\end{proof}

\begin{corollary}\label{coro:productive}
Let $\XX$ and $\YY$ be two functorial group ideals satisfying $\XX\preceq\YY$. Then $\XX\mbox{-}\YY$ is stable under finite products. Moreover, $\KK\mbox{-}\OB$ is stable under arbitrary products.
\end{corollary}
	
\begin{proposition}\label{prop:stability}
Let $\XX$ and $\YY$ be two functorial group ideals satisfying $\XX\preceq\YY$, and $f\colon G\to H$ be a morphism of $\TopGrp$.
\begin{enumerate}[(i)]
\item Suppose that $f$ is surjective and, for every $B\in\YY(H)$, there exists $B^\prime\in\YY(G)$ such that $f(B^\prime)=B$ (i.e., $f\colon(G,\YY(G))\to(H,\YY(H))$ is weakly uniformly bounded copreserving). If $G\in\XX\mbox{-}\YY$, then $H\in\XX\mbox{-}\YY$.
\item Suppose that, for every $A\in\XX(H)$, $f^{-1}(A)\in\XX(G)$ (i.e., $f\colon(G,\XX(G))\to(H,\XX(H))$ is effectively proper). If $H\in\XX\mbox{-}\YY$, then $G\in\XX\mbox{-}\YY$.
\end{enumerate}
\end{proposition}
\begin{proof}
(i) 
For every $B\in\YY(H)$, there exists $B^\prime\in\YY(G)$ such that $f(B^\prime)=B$. Since $G\in\XX\mbox{-}\YY$, $B^\prime\in\XX(G)$ and thus $B=f(B^\prime)\in\XX(G)$ because of the functoriality of $\XX$.

(ii) Let 
suppose that $H\in\XX\mbox{-}\YY$. Pick an element $B\in\YY(G)$. Then $f(B)\in\YY(H)=\XX(H)$, and so $B\subseteq f^{-1}(f(B))\in\XX(G)$, which shows the claim.
\end{proof}

\begin{corollary}\label{coro:closed_subgroup_KA}
Let $G$ be a topological group in $\KK\mbox{-}\AA$ for some functorial group ideal $\AA\succeq\KK$. Then, every closed subgroup $H$ of $G$ belongs to $\KK\mbox{-}\AA$. 
\end{corollary}
\begin{proof}
In order to apply Proposition \ref{prop:stability}(ii) to the inclusion map $i\colon H\to G$, it is enough to note that, for every $A\in\KK(G)$, $A\cap H=i^{-1}(A)\in\KK(H)$.
\end{proof}

\begin{lemma}
\label{lemma:OB_ce}
Let $G$ be a topological group and $H$ be a closed
normal
subgroup. 
\begin{enumerate}[(i)]
	\item 
(\cite[Proposition 2.35]{nicas_rosenthal})
If $H$ is compact in $G$, then $q\colon(G,\KK(G))\to(G/H,\KK(G/H))$ is a coarse equivalence.
	\item 
(\cite[Proposition 
4.37%
]{rosendal})
If $H\in\OB(G)$, then $q\colon(G,\OB(G))\to (G/H,\OB(G/H))$ is a coarse equivalence.
\end{enumerate}
\end{lemma}

\begin{corollary}\label{coro:quotient}
Let $G$ be a topological group, $H$ be a closed subgroup of $G$, and $\XX_1\preceq\KK\preceq\YY_1$ and $\XX_2\preceq\OB\preceq\YY_2$ be four functorial group ideals.
\begin{enumerate}[(i)]
\item If $H$ is compact, then $G/H\in\XX_1\mbox{-}\KK$ whenever $G\in\XX_1\mbox{-}\KK$, and $G\in\KK\mbox{-}\YY_1$ whenever $G/H\in\KK\mbox{-}\YY_1$.
\item If $H\in\OB(G)$, then $G/H\in\XX_2\mbox{-}\OB$ whenever $G\in\XX_2\mbox{-}\OB$, and $G\in\OB\mbox{-}\YY_2$ whenever $G/H\in\OB\mbox{-}\YY_2$.
\end{enumerate}
\end{corollary}
\begin{proof}
Both items easily follow from Proposition \ref{prop:stability} and Lemma \ref{lemma:OB_ce}.
\end{proof}
	
\section{Functorial group ideal diagrams on subclasses of topological groups}\label{sec:group_ideals_TopGrp}

In this section we investigate how the diagram \eqref{eq:Hasse_group_ideals}, which holds in general, changes if we consider some particular classes of (locally compact) topological groups. We are particularly interested in the 
class \LCA
of locally compact abelian groups.

Let us recall the first results in this direction due to Rosendal.
\begin{theorem}[{\cite[Corollary 
2.19%
]{rosendal}}]
	\label{prop:lc_sigma-comp}
$\KK=_{\sigmac}\OB$ (equivalently, $\KK\mbox{-}\OB\subseteq\sigmac$).
\end{theorem}
%
%
\begin{proposition}[%
{%
\begin{revty}
\end{revty}%
\cite[p.35]{rosendal}}%
]\label{prop:VFOB}
For a topological group $G$, the following properties are equivalent:
\begin{enumerate}[(i)]
	\item $G$ has no proper open subgroups;
	\item $G=\bigcup\VV(G)$;
	\item $\VV(G)=\FF(G)$;
	\item $\VV(G)=\FF(G)=\OB(G)$.
\end{enumerate}
\end{proposition}

%
%

Let us recall the following topological notions. A topological space $X$ is:
\begin{enumerate}[(i)]
\item {\em paracompact} if every open cover admits an open refinement that is locally finite;
\item {\em pseudocompact} if its image under any continuous map $f\colon X\to\mathbb R$ is bounded.
\end{enumerate}
Concerning $\KK$ and $\BB$, we have the following.
\begin{proposition}\label{prop:K_B_paracpt}
Let $G$ be a paracompact topological group.
Then $\KK(G) =\BB(G)$.
\end{proposition}
\begin{proof}
By Proposition \ref{prop:K_preceq_B}, we have $\KK(G) \subseteq \BB(G)$.
To show $\BB(G) \subseteq \KK (G)$, let $B\in \BB(G)$ and let $\Cl B$ denote the closure of $B$.
Then $\Cl B$ is paracompact, and it suffices to show that $\Cl B$ is pseudocompact by \cite[Theorem 3.10.22]{engelking}.
For this purpose, let $f\colon \Cl B \to \mathbb{R}$ be a continuous function.
Then $f$ can be extended to a continuous function $g\colon G \to \mathbb{R}$ by the Tietze-Urysohn theorem. Since $B \in \BB(G)$, $g(B)$ is bounded, which implies that $f(\Cl B)$ is bounded.
Thus $\Cl B$ is pseudocompact. 
\end{proof}

As we will be interested in locally compact groups, let us first state what holds for that class in general.
\begin{proposition}\label{prop:K_F_lc_groups}
	$\KK=_{\LC}\FF$.
\end{proposition}
\begin{proof}
By Proposition \ref{prop:K_preceq_B}, it suffices to show that $\FF \preceq_{\LC} \KK$.
Let $G\in\LC$ and $A\in\FF(G)$. Since $G$ is locally compact, there exists a relatively compact open neighbourhood $V$ of 
$e_G$.
Since $A\in\FF(G)$, there exists a finite subset $F$ of $G$ and $n\in\mathbb N$ such that $A\subseteq(FV)^n$, which concludes the proof because the latter is relatively compact.
\end{proof}
Note that Proposition \ref{prop:K_F_lc_groups} cannot be easily extended outside the realm of locally compact groups. In fact, there are even Polish groups for which that equality does not hold (e.g., every infinite-dimensional Banach space, thanks to Example \ref{ex:preceq}(ii) and Proposition \ref{prop:VFOB}, or the group $S_\infty$ of all permutations of $\mathbb N$ endowed with its Polish topology 
described in Remark \ref{rem:UB}
, see \cite[Example 
2.27%
]{rosendal}).

We now provide a result concerning the class $\Grp$
of discrete groups.

\begin{proposition}\label{prop:Grp}
	$\VV=_{\Grp}\{\{e_{-}\}\}$
, 
$[-]^{<\omega}
	=_{\Grp}\FF$, and $\OB=_{\Grp} \HH$.
\end{proposition}
\begin{proof}
	Let $G$ be a discrete group. The first claim is trivial since $\bigcup\VV(G)\subseteq\{
e_G
\}$ because the latter is open. 
	The second claim follows from Proposition \ref{prop:K_F_lc_groups} and the fact  that $[-]^{<\omega}
	=_{\Grp} \KK$.
%
%
%
To show the third claim, it suffices to show that $\HH (G) \subseteq \OB(G)$.
For this purpose, let $A \in \HH(G)$ and let $d$ be a (continuous) left-invariant pseudometric on $G$.
Define $\rho\colon G\times G \to \mathbb{R}_{\geq 0}$ by 
\begin{align*}
\rho(x,y) = 
\begin{cases}
0 & \text{if }x=y,\\
\max\{ 1, d(x,y)\} & \text{if }x\ne y. 
\end{cases}
\end{align*}
Then $\rho$ is a left-invariant compatible metric on $G$ and the identity map $id\colon G\to (G,\rho)$ is a continuous homomorphism.
Since $A \in \HH(G)$, we have $\diam_{\rho}(A) <\infty$, which implies $\diam_{d} (A)<\infty$. Thus $A \in \OB(G)$, and hence $\HH (G) \subseteq \OB(G)$.

\end{proof}

\begin{proposition}\label{prop:sigma_compact_lc_KH}
	Let $G$ be a $\sigma$-compact locally compact metrisable group. Then $\KK(G)=\HH(G)$.
\end{proposition}
\begin{proof}
By the Struble theorem \cite{struble} (see also \cite[Theorem 2.B.4]{cornulier_harpe}), $G$ has a left-invariant proper compatible metric $d$.
Thus
every non-relatively compact subset $A$ of $G$ has infinite diameter
with respect to $d$.
Then the continuous homomorphism $id\colon G\to G$ shows that $A\notin\HH(G)$.
\end{proof}

A similar result holds for arbitrary $\sigma$-compact locally compact groups, thanks to the following result.
\begin{proposition}\label{prop:sigma_compact_lc_KH_lemma}
	Let $\XX$ be a functorial group ideal such that $\KK\preceq\XX$. Suppose that the class of $\sigma$-compact locally compact metrisable groups is contained in $\KK\mbox{-}\XX$. Then so does the class of $\sigma$-compact locally compact groups. 
\end{proposition}
\begin{proof}
	Let $G$ be a $\sigma$-compact locally compact group. According to Kakutani-Koidara Theorem (\cite{kakutani_kodaira}, see also 
\cite[Theorem 3.7]{comfort} or 
\cite[Theorem 2.B.6]{cornulier_harpe}), $G$ has a compact subgroup $H$ such that $G/H$ is metrisable. Since $G/H$ is also $\sigma$-compact and locally compact, Corollary \ref{coro:quotient}(i) leads to the desired conclusion. 
\end{proof}

From Propositions \ref{prop:sigma_compact_lc_KH} and \ref{prop:sigma_compact_lc_KH_lemma} the following result trivially follows.
\begin{corollary}\label{coro:sigma_compact_lc_KH}
$\KK=_{\sigmac}\HH$.
\end{corollary}

\begin{proposition}\label{fact:fg-UC}
Let $G$ be a compactly generated group. Then $\UC(G)=\mathcal P(G)$. In particular, $\UC=_{\FG}\mathcal P
$.
\end{proposition}
\begin{proof}
%
%
%
%
Let $d$ be a continuous left-invariant ultra-pseudometric on $G$.
By Proposition \ref{prop:UC_pseudometric}, it suffices to show that $\diam_d(G) < \infty$.
Let $K$ be a compact subset of $G$ such that $G=\bigcup_nK^n$ and $e_G \in K$. 
Then $K_n \subseteq K_{n+1}$ for any $n \in \mathbb{N}$. Set $R=\diam_d(K)$, which is finite since $d$ is continuous.

We claim that $\diam_d(K^n) \leq R$ for any $n \in \mathbb{N}$ by induction on $n$.
It is obvious that $\diam_d(K)\leq R$.
Suppose that $n \geq 2$ and $\diam (K^{n-1})\leq R$.
Let $g, g' \in K^n$.
Then there exist $h,h'\in K^{n-1}$ and $k,k'\in K$ such that $g=hk$ and $g'=h'k'$. Since $d$ is a left-invariant ultra-metric,  $d(g,g') \leq \max\{ d(hk,h),d(h,h') , d(h',h'k') \} =\max\{ d(k,e_G),d(h,h'),d(e_G,k')  \}\leq R$ by induction hypothesis. Thus $\diam_d(K^n) \leq R$.

Therefore, $\diam_d (G) = \diam_{d} (\bigcup_nK^n) \leq R$, and hence $G \in \mathcal{UC}(G)$.
\end{proof}

\subsection{Results concerning abelian groups}\label{sub:group_ideals_TopAbGrp}

A quasi-homomorphism $h\colon G \to \R$ on a group $G$ is said to be \emph{homogeneous} 
if $h(g^n) = nh(g)$ for any $g \in G$ and $n \in \mathbb{Z}$.
The following fact is well-known (see \cite[Lemma 2.21]{calegari} or \cite[Proposition 3.1.2]{polterovich_rosen}).

\begin{fact}
\label{fact:homogenizsation}
For any quasi-homomorphism $f\colon G \to \mathbb{R}$, 
the map $\overline{f}\colon G \to \R$ defined by $\overline{f} (g) = \lim_{n\to \infty} 
f(g^n)/n
$, $g \in G$,  is homogeneous and closed to $f$.
\end{fact}

The map $\overline{f}$ in Fact \ref{fact:homogenizsation} is called the \emph{homogenization of $f$}.
The following holds for continuous quasi-homomorphisms on topological groups.

\begin{proposition}
\label{prop:homog:qh}
Let $G$ be a topological group and $f\colon G \to \R$ a continuous quasi-homomorphism.
Then the homogenization $\overline{f}\colon G \to \R$ of $f$ is continuous.
\end{proposition}
\begin{proof}
Take $C \geq 0$ so that $|f(gh) - f(g) -f(h)| \leq C$ for any $g, h \in G$.
For each $n \in \mathbb{N}$, define $f_n\colon G \to \R$ by  $f_n(g) = f(g^{2^n})/2^n$
for $g \in G$.
Then each $f_n$ is continuous and, for any $m,n \in \mathbb{N}$ with $m < n$, we have 
$\sup \{ | f_n(g) -f_m(g)|  \mid g \in G \} \leq {C/2^m}
$ (see the proof of \cite[Lemma 2.21]{calegari}).
Thus the sequence $\{f_n\}$ converges to $\overline{f}$ uniformly, and hence $\overline{f}$ is continuous.
\end{proof}

The following fact is well-known (see \cite[Proposition 3.1.4]{polterovich_rosen} or \cite[Proposition 2.65]{calegari}).
\begin{fact}\label{fact:abel_homog_qh}
Every homogeneous quasi-homomorphism on an abelian group is a homomorphism.
\end{fact}

The following result 
concern the case of discrete abelian groups.

\begin{proposition}\label{prop:QH_HH_Ab}
For every 
topological abelian group 
$G$, $\QH_{\R} (G) = \HH_{\R}(G)$.
\end{proposition}
\begin{proof}
Let $G$ be 
a topological abelian 
group, and $A\in\HH_{\R}(G)$. Pick a
continuous
quasi-homomorphism $f\colon G\to\R$. We claim that $\diam (f(A))<\infty$, and so $A\in\QH_{\R}(G)$. Thanks to 
Proposition \ref{prop:homog:qh} and Fact \ref{fact:abel_homog_qh},
there exists a 
continuous
homomorphism $g\colon G\to\R$ close to $f$. Since $A\in\HH_{\R}(G)$, $\diam (g(A))<\infty$. Hence, according to Remark \ref{rem:bounded_subsets_close_maps}, $\diam (f(A))<\infty$.
\end{proof}



Let us recall the following notion coming from group theory. An abelian group is called {\em divisible} if, for every $g\in G$ and $n\in\mathbb N$, there exists $x\in G$ such that $nx=g$. It is folklore (see, for example, \cite[%
Theorem 23.1%
]{Fuc}) that every divisible abelian group $G$ is of the following form:
$$\mathbb Q^{(r_0(G))}\oplus\bigg(\bigoplus_{p>0\text{ prime}}(\mathbb Z_{p^{\infty}})^{(r_p(G))}\bigg),$$
where $\mathbb Z_{p^{\infty}}$ is the $p$-Pr\"ufer group, and $r_0(G)$ and $r_p(G)$ are the free-rank and the $p$-rank of $G$, respectively. In particular, every divisible abelian group is a direct sum of countable groups.

For discrete abelian groups, the following result can be shown.
\begin{proposition}
	Let $G$ be a discrete abelian group. Then the following properties are equivalent:
	\begin{enumerate}[(i)]
		\item $G$ is torsion;
		\item $G\in\HH_{\R}(G)$;
		\item $G\in\QH_{\R}(G)$.
	\end{enumerate}
\end{proposition}
\begin{proof}
	Items (ii) and (iii) are equivalent because of Proposition \ref{prop:QH_HH_Ab}
	. Moreover, the implication (i)$\to$(ii) is trivial. Conversely, let us prove (ii)$\to$(i). Suppose that the group $G$ is not torsion and consider its divisible hull $D(G)$ containing a copy of $G$ as an essential subgroup (i.e., for every non-trivial subgroup $H$ of $D(G)$, $H\cap G\supsetneq\{0\}$%
	, see,  for example \cite[Lemma 24.3 and Theorem 24.4]{Fuc}%
	). For the sake of simplicity, let us denote 
	the copy 
	again by $G$. Since $G$ is not torsion, $D(G)=\mathbb Q\times D$ for some divisible group $D$. Note that $G\cap (\mathbb Q\times\{0\})\neq\{0\}$ because $G$ is an essential subgroup of $D(G)$. Denote by $f\colon D(G)\to \mathbb Q$ the canonical projection and by $i\colon\mathbb Q\to\R$ the inclusion map. Then the homomorphism $(i\circ f)\restriction_G\colon G\to\R$ has unbounded image and thus $G\notin\HH_{\R}(G)$.
\end{proof}

A topological group $G$ is said to be \emph{a-T-menable} if there exists a continuous metrically proper affine isometric action $\alpha$ of $G$ on a Hilbert space $H$, where $\alpha$ is said to be \emph{metrically proper} if for every bounded subset $B$ of $H$, the set $\{g \in G \mid \alpha(g) B \cap B \ne \emptyset\}$ is relatively compact.
Note that every discrete a-T-menable group is countable (see \cite[Lemma 5.16]{mislin}).
	
For example, every countable discrete amenable group is a-T-menable (see \cite[%
Theorem 18.50,
Corollary 19.43 and Theorem 
19.50%
]{drutu_kapovich}). 
Since every abelian group is amenable (see \cite[Corollary 18.54]{drutu_kapovich}), every countable discrete abelian group is a-T-menable.  

\begin{proposition}
\label{prop:a-T-menable}
Let $G$ be a countable discrete a-T-menable group.
Then $[G]^{<\omega} = \mathcal{FH}(G)$.
In particular, $[-]^{<\omega}
=_{\countAb}\FH$.
\end{proposition}
\begin{proof}
The inclusion $[G]^{<\omega} \subseteq  \mathcal{FH}(G)$ is always valid.
Let $\alpha$ be a  continuous metrically proper affine isometric action $\alpha$ of $G$ on a Hilbert space $H$. 
To show $\mathcal{FH}(G) \subseteq  [G]^{<\omega}$, let $A \in \mathcal{FH}(G)$. Set $B= \alpha (A) 0 \cup \{ 0\}$, where $0$ is the zero element in $H$, and $C=\{ g \in G\mid \alpha (g) B \cap B \ne \varnothing\}$. Then $A \subseteq C$. Since $B$ is bounded and $\alpha$ is metrically proper, $C$ is relatively compact. Thus $C$ is finite since $G$ is discrete. This and $A \subseteq C$ imply that $A \in \mathcal{K}(G)$, and hence $\mathcal{FH}(G) \subseteq  \mathcal{K}(G)$.
\end{proof}

By an argument similar to the proof of \cite[Lemma 4.6]{tessera-valette}, we also have the following.

	\begin{proposition}\label{prop:from_countAb_to_Ab}
	Let $\AA$ be a functorial group ideal satisfying $[-]^{<\omega}
		\preceq_{\Ab}\AA$. If $[-]^{<\omega}
		=_{\countAb}\AA$, then $[-]^{<\omega}
		=_{\Ab}\AA$.
	\end{proposition}
\begin{proof}
	Let $G$ be a discrete abelian group.
	To show that $\AA(G) \subseteq [G]^{<\omega}$,  let $A$ be an infinite subset of $G$.
	Using the idea in the proof of \cite[Theorem 2.5.2]{rudin}, we show $A \notin \AA(G)$.
	Take a countably infinite subset $C$ of $A$ and let $\Gamma$ be the subgroup of $G$ generated by $C$.
	Then there exists an injective homomorphism $f$ on $\Gamma$ to some countable divisible group $D$ (see \cite[Theorem 2.5.1(a)]{rudin}), 
	which can be extended to a homomorphism $h\colon G \to D$ (see \cite[Theorem 2.5.1(b)]{rudin}). Hence, in particular, $h(A)=f(A)\notin[D]^{<\omega}=\AA(D)$. Since $\AA$ is functorial, $A\notin\AA(G)$.
	\end{proof}

By Propositions \ref{prop:sigma_compact_lc_KH}, \ref{prop:a-T-menable} and \ref{prop:from_countAb_to_Ab}, we have the following.

\begin{corollary}\label{coro:Ab}
 $[-]^{<\omega}
 =_{\Ab}\HH=_{\Ab}\FH$. 
\end{corollary}

In order to derive results concerning the class $\LCA$, let us first recall the following structure theorem.

\begin{theorem}[%
see
{\cite[Theorem 3.3.%
\begin{revty}
10%
\end{revty}%
]{DikProSto}}
or {\cite[Theorem 14.2.18]{AusDikGio}}%
]\label{thm:lca_structure}
	Let $G$ be a locally compact abelian 
	group. Then $G$ has the form $\R^n\times G_0$ where $n\in\mathbb N$, and $G_0$ is a closed subgroup of $G$ having a compact open subgroup.
\end{theorem}

\begin{proposition}\label{prop:from_Ab_to_LCA}
Let $\AA$ be a functorial group ideal satisfying $\KK\preceq_{\LCA}\AA$. If $\KK=_{\Ab}\AA$ and $\KK(\R^n)=\AA(\R^n)$, for every $n\in\mathbb N$, then $\KK=_{\LCA}\AA$.
\end{proposition}
\begin{proof}
	Let $G$ be a locally compact abelian group. By Theorem \ref{thm:lca_structure}, we may assume $G=\mathbb{R}^n\times G_0$ for some $n \in \mathbb{N}$ and some closed subgroup $G_0$ of $G$ having a compact open subgroup $K$. According to Corollary \ref{coro:productive}, $\KK(G)=\AA(G)$ if and only if $\KK(\R^n)=\AA(\R^n)$ and $\KK(G_0)=\AA(G_0)$. Consider the quotient homomorphism $q\colon G_0\to G_0/K$, and note that $G_0/K$ is discrete and thus $\KK(G_0/K)=\AA(G_0/K)$. 
	Finally, thanks to Corollary \ref{coro:quotient}(i), $\KK(G_0)=\AA(G_0)$ since $K$ is compact.
\end{proof}

In order to apply Proposition \ref{prop:from_Ab_to_LCA} to some specific functorial group ideals, we 
	first focus on the group $\R^n$.
\begin{proposition}\label{prop:Rn}
	For every $n\in\mathbb N$, $\KK(\R^n)=\VV(\R^n)=\HH_{\R}(\R^n)\subsetneq\UC(\R^n)=\mathcal P(\R^n)$.
\end{proposition}
\begin{proof}
	Let $A\subseteq\R^n$ be a non-relatively compact subset of $\R^n$. Then $A$ is unbounded and thus there exists a projection $p_j\colon\R^n\to\R$, where $j\in\{1,\dots,n\}$, such that $p_j(A)$ is unbounded. Since projections are homomorphisms, $A\notin\HH_{\R}(\R^n)$. Hence $\KK(\R^n)=\HH_{\R}(\R^n)$. Moreover, Proposition \ref{prop:VFOB} implies that $\VV(\R^n)=\OB(\R^n)=\KK(\R^n)$. Finally, the last equality follows from Proposition \ref{fact:fg-UC}.
\end{proof}

By Corollary \ref{coro:Ab} and Propositions \ref{prop:from_Ab_to_LCA} and \ref{prop:Rn}, we have the following corollary, which implies Theorem \ref{thm:K_LCA_OB}.

\begin{corollary}\label{coro:LCA}
 $\KK=_{\LCA}\HH=_{\LCA}\FH$.
 In particular, $\KK=_{\LCA}\OB$.
\end{corollary}

For the class $\FGAb$ of finitely generated discrete abelian group, we have the following.

\begin{proposition}\label{prop:FGAb}
$[-]^{<\omega}
=_{\FGAb}\HH_{\R}$.	
\end{proposition}
\begin{proof}
Let $G$ be a finitely generated 
discrete
abelian group, and $A\subseteq G$ be an infinite subset. 
We may assume $G=\mathbb{Z}^n\times F$ for some $n \in \mathbb{N}$ and some finite group $F$ by the fundamental theorem of finitely generate abelian group.
We want to construct a (necessarily continuous) homomorphism $f\colon G\to\mathbb R$ such that $f(A)$ is unbounded. Since $A$ is infinite, 
by the pigeonhole principle,
there exists a projection $p\colon\mathbb Z^n\times F\to\mathbb Z$ such that $p(A)$ is infinite. The composite $i\circ p\colon G\to\mathbb R$,
where $i$ is the canonical inclusion of $\mathbb Z$ into $\mathbb R$, is the desired homomorphism.
\end{proof}




\section{Applications of Theorem \ref{thm:K_LCA_OB}}
\label{sec:Appl}

\subsection{Non-embeddability of infinite dimensional Banach spaces into a product of locally compact groups}
\label{sec:embBanach}

Let $X$
be a Banach space, which is assumed to be a topological group with the vector addition and the topology induced by the norm.  In \cite[Corollary 2.20]{rosendal}, it was proved that $\OB(X)$ coincides with the family of all subsets which are bounded with respect to the norms of $X$. Thus, if $X$ is infinite-dimensional, then $X \notin \KK\mbox{-}\OB$ since the closed  unit ball belongs to $\OB(X)\setminus\KK(X)$.
By applying this fact and Theorem \ref{thm:K_LCA_OB} with Proposition \ref{prop:arbitr_productive} and Corollary \ref{coro:closed_subgroup_KA}, we have the following.

\begin{theorem}\label{thm:Banach_space_not_in_prod_lc}
Let $X$ be an infinite-dimensional Banach space. Then $X$ 
cannot be embedded in any product of locally compact groups as a topological group.
\end{theorem}
\begin{proof}
Suppose by contradiction that there is a family $\{G_i\}_{i\in I}$ of locally compact groups such that $X$ embeds into $
\begin{revty}
\end{revty}
\Pi_iG_i$ as a topological group. Denote by $\iota\colon X\to 
\Pi_iG_i
$ that embedding. 
For each $i \in I$, let $H_i$ be the closure of $p_i(\iota(X))$. 
Then each $H_i$ is a locally compact abelian group and $\iota (X) \subseteq \Pi_i H_i$.
As, for every $i\in I$, $
H_i
\in\KK\mbox{-}\OB$ (Theorem \ref{thm:K_LCA_OB}), Corollary \ref{coro:productive} implies that $
\Pi_iH_i
\in\KK\mbox{-}\OB$. 
Since the corestriction $\iota\colon X\to\iota(X)$ is a topological isomorphism, $\iota(X)$ is complete, and it is a closed subgroup 
$\Pi_iH_i$ 
(see, for example, \cite[Proposition 7.1.8]{AusDikGio}). 
Thus, 
according to Corollary \ref{coro:closed_subgroup_KA}, $\iota(X)\in\KK\mbox{-}\OB$. We conclude that $X\in\KK\mbox{-}\OB$ by applying the functoriality of both $\KK$ and $\OB$, which contradicts the above fact.	
\end{proof}

\subsection{Metrisability of the left-coarse struture of the dual group of a locally compact abelian group}
\label{sub:question_pon}
Let $G$ be a locally compact abelian group and let $\widehat{G}$ denote the dual group (or the character group) of $G$ (see, for example, \cite[Definition 36]{pontryagin} or \cite[Section 13.1]{AusDikGio}).
We say that a coarse group $(G,\II)$ is {\em metrisable} if there exists a left-invariant metric $d$ on it such that $\E_{\II}=\E_d$.

In \cite[Theorem 3.15]{dikranjan_zava_Pon}, it was proved that $G$ is metrisable as a topological groups if and only if the coarse groups $(\widehat G,\KK(\widehat G))$ is metrisable, and that, if $G$ is metrisable as a topological groups, then the coarse group $(\widehat G,\OB(\widehat G))$ is metrisable. 
It was also asked in \cite[Question 5.6]{dikranjan_zava_Pon} whether the converse of the latter statement holds.
Since the dual group $\widehat{G}$ is a locally compact abelian group (see, for example, \cite[Theorem 36]{pontryagin} or \cite[Corollary 13.1.3]{AusDikGio}), Theorem \ref{thm:K_LCA_OB} implies $\KK(\widehat G)= \OB(\widehat G)$, which answers the question  affirmatively and we have the following.

\begin{corollary}\label{coro:question_pon}
	Let $G$ be a locally compact abelian group. Then the following properties are equivalent:
	\begin{enumerate}[(i)]
		\item $G$ is metrisable as a topological group;
		\item $(\widehat G,\KK(\widehat G))$ is metrisable;
		\item $(\widehat G,\OB(\widehat G))$ is metrisable.
	\end{enumerate}
\end{corollary}

\section{Hasse diagrams of the partial orders induced by classes of topological groups}\label{sec:Hasse_group_ideals}

For a better understanding of the results stated in 
Section \ref{sec:group_ideals_TopGrp},
let us represent the Hasse diagrams of the partial orders induced by classes of topological groups.

According to Proposition \ref{prop:K_F_lc_groups}, the Hasse diagram of $(\mathfrak X,\preceq_{\LC})$ is:
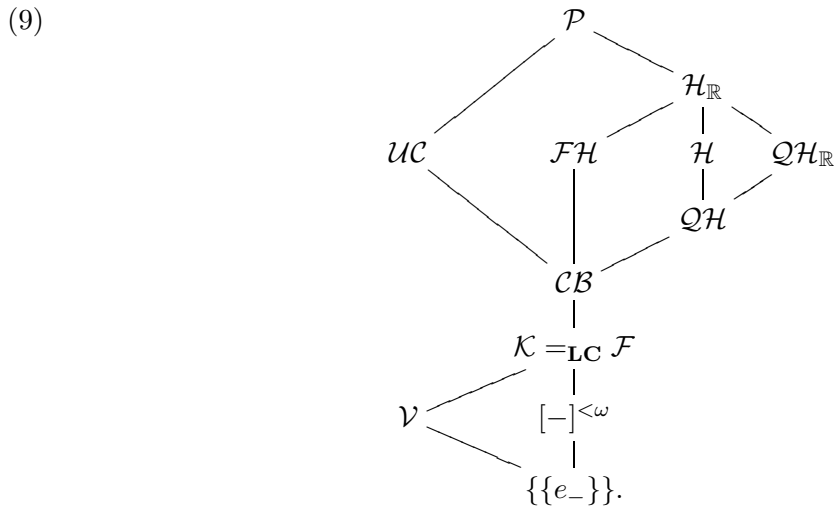
\begin{equation}\label{eq:Hasse_group_ideals_LC}
	\xymatrix@-1.2pc{& &  \mathcal P
		\ar@{-}[dr]\ar@{-}[ddll] & &\\
		& & &\HH_{\R}\ar@{-}[dl]\ar@{-}[dr]\ar@{-}[d] &\\
		\UC\ar@{-}[ddrr] &  & \FH\ar@{-}[dd] &\HH\ar@{-}[d] &\QH_{\R}\ar@{-}[dl]\\
		& &  & \QH\ar@{-}[dl] &\\
		& & \OB \ar@{-}[d]& &\\
		& & \KK=_{\LC}\FF
		\ar@{-}[d] & & \\
		\VV\ar@{-}[drr]\ar@{-}[urr] &  & [-]^{<\omega}
		\ar@{-}[d] & &\\
		& & \{\{e_{-}\}\}
		. & &}
\end{equation}

As for $(\mathfrak X,\preceq_{\Grp})$, thanks to Proposition \ref{prop:Grp}, we have:
\begin{equation}\label{eq:Hasse_group_ideals_Grp}
 	\xymatrix@-1.2pc{ & \mathcal P
 		\ar@{-}[d]\ar@{-}[ddl] & \\
 		 & \HH_{\R}\ar@{-}[d]\ar@{-}[dr]& \\
 		\UC\ar@{-}[dr]   & \FH\ar@{-}[d]  &\QH_{\R}\ar@{-}[dl]\\
 		&  \OB =_{\Grp} \HH\ar@{-}[d] & \\
 		&  [-]^{<\omega}
 		=_{\Grp}\FF
 		\ar@{-}[d] & \\
 		&  
 \{\{e_{-}\}\}
 =_{\Grp}\VV.
  & }
\end{equation}

Corollary \ref{coro:sigma_compact_lc_KH} implies that $(\mathfrak X,\preceq_{\sigmac})$ is:
\begin{equation}\label{eq:Hasse_group_ideals_sigmac}
	\xymatrix@-1.2pc{
&  \mathcal P
\ar@{-}[d]\ar@{-}[ddl] &\\
&  \HH_{\R}\ar@{-}[d]\ar@{-}[dr] & \\	
 \UC\ar@{-}[dr] & \FH\ar@{-}[d] & \QH_{\R}\ar@{-}[dl] \\		
 & \KK=_{\sigmac}\HH\ar@{-}[dr]\ar@{-}[dl] &\\
 \VV\ar@{-}[dr] & & [-]^{<\omega}
 \ar@{-}[dl]  \\
 & \{\{e_{-}\}\}
 . & 
}
\end{equation}

Combining the diagrams \eqref{eq:Hasse_group_ideals_Grp} and \eqref{eq:Hasse_group_ideals_sigmac}, we obtain $(\mathfrak X,\preceq_{\countG})$:
\begin{equation}\label{eq:Hasse_group_ideals_countG}
	\xymatrix@-1.2pc{
		&  \mathcal P
		\ar@{-}[d]\ar@{-}[ddl] &\\
		&  \HH_{\R}\ar@{-}[d]\ar@{-}[dr] & \\	
		\UC\ar@{-}[dr] & \FH\ar@{-}[d] & \QH_{\R}\ar@{-}[dl] \\		
		& [-]^{<\omega}
		=_{\countG}\HH\ar@{-}[d] &\\
		& 
\{\{e_{-}\}\}
=_{\countG}\VV.
 & 
	}
\end{equation}

Proposition \ref{fact:fg-UC} implies the following situation for $(\mathfrak X,\preceq_{\FG})$:
\begin{equation}\label{eq:Hasse_group_ideals_FG}
	\xymatrix@-1.2pc{
		&  \UC=_{\FG}\mathcal P
		\ar@{-}[d]&\\
		&  \HH_{\R}\ar@{-}[dl]\ar@{-}[dr] & \\	
		\FH\ar@{-}[dr] & & \QH_{\R}\ar@{-}[dl] \\		
		& [-]^{<\omega}
		=_{\FG}\HH\ar@{-}[d] &\\
		& 
\{\{e_{-}\}\}
=_{\FG}\VV.
& 
	}
\end{equation}

Let us now represent the situation for abelian groups. 
Proposition \ref{prop:QH_HH_Ab}, Corollary \ref{coro:LCA}
and \eqref{eq:Hasse_group_ideals_LC} imply that the Hasse diagram of  $(\mathfrak X,\preceq_{\LCA})$ is: 

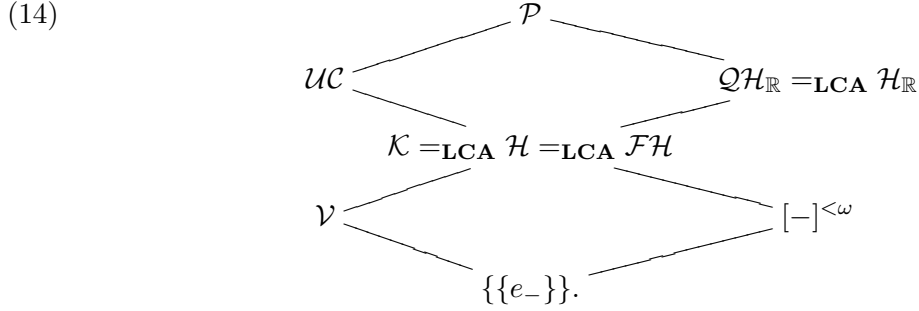
\begin{equation}\label{eq:Hasse_group_ideals_lca}
	\xymatrix@-1.2pc{
&  \mathcal P
\ar@{-}[dr]
\ar@{-}[dl] 
&\\
\UC\ar@{-}[dr]
 &  & \QH_{\R}=_{\LCA} \HH_{\R}
\ar@{-}[dl]\\
 & \KK=_{\LCA}\HH
=_{\LCA}\FH
\ar@{-}[dr]\ar@{-}[dl] &\\
 \VV\ar@{-}[dr] & & [-]^{<\omega}
 \ar@{-}[dl]  \\
 & \{\{e_{-}\}\}
 . & 
}
\end{equation}

The Hasse diagram of $(\mathfrak X,\preceq_{\sigmacAb})$ coincides with \eqref{eq:Hasse_group_ideals_lca}.
Combining the diagrams \eqref{eq:Hasse_group_ideals_Grp} and \eqref{eq:Hasse_group_ideals_lca}, we obtain $(\mathfrak X,\preceq_{\Ab})$: 

\begin{equation}\label{eq:Hasse_group_ideals_Ab}
	\xymatrix@-1.2pc{
&  \mathcal P
\ar@{-}[dr]
\ar@{-}[dl] 
&\\
\UC\ar@{-}[dr]
 &  & \QH_{\R}=_{\Ab} \HH_{\R}
\ar@{-}[dl]\\
		& [-]^{<\omega}
		=_{\Ab}\HH
=_{\Ab}\FH
\ar@{-}[d] &\\
		& \{\{e_{-}\}\}
		=_{\Ab}\VV. & 
	}
\end{equation}

The Hasse diagram of $(\mathfrak X,\preceq_{\countAb})$ coincides with \eqref{eq:Hasse_group_ideals_Ab}.
Finally, according to Proposition \ref{prop:FGAb}, $(\mathfrak X,\preceq_{\FGAb})$ is:
\begin{equation}\label{eq:Hasse_group_ideals_FGAb}
	\xymatrix@-1.2pc{
		 \UC=_{\FGAb}\mathcal P
		 \ar@{-}[d]\\
		 [-]^{<\omega}
		 =_{\FGAb}\HH_{\R}\ar@{-}[d]\\
\{\{e_{-}\}\}
=_{\FGAb}\VV.
	}
\end{equation}

\subsection{Counterexamples and questions}\label{sub:counterexamples_questions}

The remaining part of the section is devoted to provide counterexamples showing how accurate the Hasse diagrams \eqref{eq:Hasse_group_ideals}, and \eqref{eq:Hasse_group_ideals_LC}--\eqref{eq:Hasse_group_ideals_FGAb} are.

We discuss the diagrams taking into account Fact \ref{fact:induced_partial_order_reverted}. In particular, we use the following trivial application: for a pair of functorial group ideals $\XX,\YY\in\mathfrak X$ and two subclasses $\mathbb G_1\subseteq\mathbb G_2\subseteq\TopGrp$, if $\XX\not\preceq_{\mathbb G_1}\YY$, then $\XX\not\preceq_{\mathbb G_2}\YY$.

For that reason, let us start with \eqref{eq:Hasse_group_ideals_FGAb}, and the following example shows that it precisely represent the Hasse diagram of $(\mathfrak X,\preceq_{\FGAb})$.
\begin{example}
\label{ex:FGAb}
If $G$ is an infinite finitely generated abelian group (e.g., $G=\mathbb Z$), then $\{\{0\}\}\subsetneq [G]^{<\omega}\subsetneq\mathcal P(G)$.
\end{example}

Let us now consider the 
classes $\countAb$ and $\Ab$, and so the diagram \eqref{eq:Hasse_group_ideals_Ab}.
\begin{example}\label{ex:countAb}
Consider 
the countable torsion abelian group $G=\bigoplus_{\mathbb Z}\mathbb Z_2$. 
Every homomorphism $f\colon G\to\R$ has to send $G$ into the torsion subgroup of $\R$, which is the singleton $\{0\}$. Hence $G\in\HH_{\R}(G)$.
Next we
want to show that $G\notin\UC(G)$.
For any two sequences $\overline x=(x_n)_{n\in\mathbb Z},\overline y=(y_n)_{n\in\mathbb Z}
\in G
$, define 
$$L(\overline x,\overline y)=\max\{n\in\mathbb Z\mid\forall m\leq n, x_m=y_m\},\quad\text{and}\quad d(\overline x,\overline y)=2^{-L(\overline x,\overline y)},$$
with the assumption that, if $L(\overline x,\overline y)=\infty$, then $d(\overline x,\overline y)=0$. 

We claim that $d$ is a left-invariant ultra-metric on $G$. 
The function $d$ assumes only finite values as every sequence $\overline x\in G$ has $1$ only in finitely many entries. Clearly, $d(\overline x,\overline y)=0$ if and only if $\overline x=\overline y$, and $d$ is symmetric. Let now $\overline x,\overline y,\overline z\in G$. Suppose, without loss of generality, that $\overline x\neq\overline y$. It is easy to see that either $L(\overline x,\overline z)\leq L(\overline x,\overline y)$ or $L(\overline y,\overline z)\leq L(\overline x,\overline y)$. Because of this observation, we have that either $d(\overline x,\overline z)\geq d(\overline x,\overline y)$ or $d(\overline y,\overline z)\geq d(\overline x,\overline y)$, and so $d$ is an ultra-metric. To finally prove invariance, note that, for every $(x_n)_n,(y_n)_n,(z_n)_n\in G$ and every $m\in\mathbb Z$, $x_m=y_m$ if and only if $x_m+z_m=y_m+z_m$.

Since $\diam_{d} (G)$ is infinite, by Proposition \ref{prop:UC_pseudometric}, we have $G\notin\UC(G)$.

Thus $\HH_{\R}(G) \not\subseteq \UC(G)$, and hence $\HH_{\R} \not\preceq_{\countAb} \UC$, which implies 
$\HH_{\R} \not\preceq_{\countAb} [-]^{<\omega}$ and $\mathcal{P} \not\preceq_{\countAb} \UC$.
\end{example}

Let us now consider the classes $\sigmacAb$ and $\LCA$.
\begin{example}\label{ex:sigmacAb_LCA}
According to Proposition \ref{prop:Rn}, $[\R]^{<\omega}\subsetneq
\KK(\R)=\VV(\R)
$, and thus $\VV\not\preceq_{\sigmacAb}[-]^{<\omega}
$%
, which implies 
$\VV\not\preceq_{\sigmacAb}\{\{e_{-}\}\}$
and $\KK\not\preceq_{\sigmacAb}[-]^{<\omega}$.
\end{example}

Let us now loose the abelianity request and start with the diagram \eqref{eq:Hasse_group_ideals_FG}.

\begin{example}\label{ex:FG}
\begin{enumerate}[(i)]
	\item Let 
	$F_2$
	be the free group generated by two elements $a$ and $b$ endowed with the discrete topology. 
	
	First, we show $\HH_{\R}(F_2) \not\subseteq \QH_{\R}(F_2)$.		
	Consider the subset $A=\{(aba^{-1}b^{-1})^n\mid n\in\mathbb N\}$. Let $f\colon 
	F_2
	\to\mathbb R$ be a homomorphism. Since $\mathbb R$ is abelian, $f(A)=\{0\}$, and thus $A\in\HH_{\R}(
	F_2
	)$. 
	Consider the map $f\colon F_2	\to\mathbb R$ due to R. Brooks (\cite[%
\S 3%
]{Bro}
) 
	defined as follows:
	for every $x\in 
	F_2
	$, $f(x)$ is the number of non-overlapping occurrences of the subword $aba^{-1}b^{-1}$ minus the ones of $(aba^{-1}b^{-1})^{-1}$ in the reduced word representing $x$. Then $f$ is not a homomorphism, but it is a quasi-homomorphism 
	(\cite[Proposition 2.30]{calegari}).
	Indeed, $f(aba^{-1}b^{-1}) =1 \ne 0 =f(ab) + f(a^{-1}b^{-1})$ and,
	 for every $x,y\in 
	F_2
	$, $\lvert f(xy)-f(x)-f(y)\rvert\leq 1$ (at most one occurrence of $aba^{-1}b^{-1}$ or $(aba^{-1}b^{-1})^{-1}$ can be created or deleted by the concatenation of the end of $x$ and the beginning of $y$). Moreover, $f(A)=\mathbb N$, which is unbounded, and this shows that $A\notin\QH_{\R}(
	F_2
	)$. Thus $\HH_{\R}\not\preceq_{\FG}\QH_{\R}$.

To show $\QH_{\R}(
F_2
) \not\subseteq [
F_2
]^{<\omega}$, let $B=\{a^nba^{-n} \mid n \in \mathbb{N} \}$. Then $B \notin [
F_2
]^{<\omega}$. To see $B \in \QH_\R$, let $f\colon G \to \R$ be a quasi-homomorphism and $\overline{f}\colon 
F_2
\to \R$ its homogenization (see Fact \ref{fact:homogenizsation}). Take $C \geq 0$ satisfying $|\overline{f} (gh) - \overline{f}(g) - \overline{f}(h) | \leq C$ for any $g,h \in 
F_2
$.
Since $\overline{f}$ is homogeneous,  
$\overline{f}(a^nba^{-n}) \leq \overline{f}(a^n) + \overline{f}(b) + \overline{f}(a^{-n}) +2C =n\overline{f}(a) +\overline{f}(b) -n \overline{f}(a)+2C = \overline{f}(b) +2C$ for every $n \in \mathbb{N}$, and thus $\diam (\overline{f}(B) )<\infty$.
Since $f$ and $\overline{f}$ are close,  according to Remark \ref{rem:bounded_subsets_close_maps}, we have $\diam (f(A))<\infty$, 
and hence $A \in \QH_\R(
F_2
)$. Thus
$\QH_{\R}(
F_2
) \not\subseteq [
F_2
]^{<\omega}$.

Since 
$F_2$
is a-T-menable (see \cite[Proposition 6.2.6]{nowak-yu}), 
$[
F_2
]^{<\omega}= \FH(
F_2
)$ by 
Proposition \ref{prop:a-T-menable}.
Therefore
 $\QH_{\R}\not\preceq_{\FG}[-]^{<\omega}
 $
and $\QH_{\R}\not\preceq_{\FG}\FH$.
The latter implies $\HH_{\R}\not\preceq_{\FG}\FH$.
	\item Consider the finitely generated group $SL_3(\mathbb Z)$
	with the discrete topology.
	Then $
\FH(SL_3(\mathbb Z))=\mathcal P(SL_3(\mathbb Z))$ by 
Remark \ref{rem:prop.FH}. Thus $\FH\not\preceq_{\FG}[-]^{<\omega}
$. 
\end{enumerate}
\end{example}

As for the class $\Grp$, we have the following example.
\begin{example}\label{ex:Grp}
Let $
\Symd
$ be the group of all permutation of $\mathbb N$ endowed with the discrete topology. 
	Then $[
\Symd
]^{<\omega} \subsetneq \mathcal{P} (
\Symd
)$.
	According to \cite[Theorems 5 and 6 and Lemma 10]{bergman},
$\OB(
\Symd
)=\mathcal P(
\Symd
)$ (see 
also
\cite[Example 
2.26%
]{rosendal}). Hence $\OB\not\preceq_{\Grp}
	[-]^{<\omega}
$.	

	

\end{example}

Finally, we consider the relationships of groups ideals in $\mathfrak{X}$ for the class $\TopGrp$.

\begin{example}\label{ex:preceq}
\begin{enumerate}[(i)]
	\item 
	Let $G$ be a pseudocompact non-compact group such as the $\Sigma$-product of uncountably many copies of a nontrivial compact group 
\begin{revty}
\end{revty} 
(see, for example, \cite[Example 1.2.9]{arhangelskii_tkachenko}).
	Then $G \in \BB(G) \setminus \KK(G)$. 
	Hence $\BB\not\preceq\KK$.%
	\item Let 
	$X$
	be an infinite-dimensional Banach space.
	As we have seen in Section \ref{sec:embBanach}, $\OB(X)$ consists of the subsets that are norm-bounded and $\KK(X)\subsetneq\OB(X)$.
	Since 
	$X$
	has no proper open subgroup, Proposition \ref{prop:VFOB} implies that $\VV(
	X
	)=\FF(
	X
	)=\OB(
	X
	)$. 
	Since $X$ is paracompact (being metrizable), Proposition \ref{prop:K_B_paracpt} implies $\KK(G) =\BB(G)$. Thus we have that $\VV\not\preceq\KK$, $\VV\not\preceq\BB$ and $\FF\not\preceq\BB$. 
\end{enumerate}
\end{example}

The following questions remain unsolved:
\begin{question}\label{q}
\begin{enumerate}[(i)]
\item For $\mathbb{G} \in \{\FG, \countG, \sigmac, \Grp, \LC, \TopGrp \}$,
 is it true that $\FH\not\preceq_{\mathbb{G}}\QH_{\R}$?
\item For $\mathbb{G} \in \{\countG, \sigmac, \Grp, \LC, \TopGrp \}$, is it true that $\FH\not\preceq_{
\mathbb{G}
}\UC$?
\item For $\mathbb{G} \in \{ \LC, \TopGrp \}$, is it true that $\HH\not\preceq_{\mathbb{G}}\UC$, $\HH\not\preceq_{\mathbb{G}}\FH$, $\HH\not\preceq_{\mathbb{G}}\QH$, 
$\HH\not\preceq_{\mathbb{G}}\QH_{\R}$, 
$\QH\not\preceq_{\mathbb{G}}\UC$, $\QH\not\preceq_{\mathbb{G}}\FH$ and $\QH\not\preceq_{\mathbb{G}}\OB$?
\end{enumerate}
\end{question}

\end{document}